\documentclass[twoside,a4paper,reqno,11pt]{amsart} 
\usepackage[top=26mm,right=26mm,bottom=26mm,left=26mm]{geometry}
\usepackage{microtype}
\usepackage{amsfonts, amsmath, amssymb, amsbsy, mathrsfs, bm, latexsym, stmaryrd, hyperref, array, enumitem, textcomp, url}
\usepackage{framed}
\usepackage[table]{xcolor}

\renewcommand{\a}{\alpha}
\renewcommand{\b}{\beta}

\renewcommand{\l}{\lambda}
 
\newcommand{\s}{\sigma}
\renewcommand{\O}{\Omega}

\newcommand{\la}{\langle}
\newcommand{\ra}{\rangle}
\newcommand{\leqs}{\leqslant}
\newcommand{\geqs}{\geqslant}

\newcommand{\vs}{\vspace{3mm}}

\makeatletter
\newcommand{\imod}[1]{\allowbreak\mkern4mu({\operator@font mod}\,\,#1)}
\makeatother

\theoremstyle{plain}
\newtheorem{theorem}{Theorem} 
 
\newtheorem{corol}[theorem]{Corollary}
\newtheorem{thm}{Theorem}[section] 
\newtheorem{lem}[thm]{Lemma}
\newtheorem{prop}[thm]{Proposition} 

\newtheorem{cor}[thm]{Corollary} 
\newtheorem*{theorem*}{Theorem} 
\newtheorem*{conj*}{Conjecture}

\theoremstyle{definition}
\newtheorem{rem}[thm]{Remark}

\newtheorem{remk}{Remark}

\begin{document}

\title[Topological generation by unipotent elements]{On the topological generation of exceptional groups \\ by unipotent elements}

\author{Timothy C. Burness}
\address{T.C. Burness, School of Mathematics, University of Bristol, Bristol BS8 1UG, UK}
\email{t.burness@bristol.ac.uk}

\date{\today} 

\begin{abstract}
Let $G$ be a simple algebraic group of exceptional type over an algebraically closed field  of characteristic $p \geqs 0$ which is not algebraic over a finite field. Let $\mathcal{C}_1, \ldots, \mathcal{C}_t$ be non-central conjugacy classes in $G$. In earlier work with Gerhardt and Guralnick, we proved that if $t \geqs 5$ (or $t \geqs 4$ if $G = G_2$), then there exist elements $x_i \in \mathcal{C}_i$ such that $\la x_1, \ldots, x_t \ra$ is Zariski dense in $G$. Moreover, this bound on $t$ is best possible. Here we establish a more refined version of this result in the special case where $p>0$ and the $\mathcal{C}_i$ are unipotent classes containing elements of order $p$. Indeed, in this setting we completely determine the classes $\mathcal{C}_1, \ldots, \mathcal{C}_t$ for $t \geqs 2$ such that $\la x_1, \ldots, x_t \ra$ is Zariski dense for some $x_i \in \mathcal{C}_i$.      
\end{abstract}

\maketitle

\section{Introduction}\label{s:intro}

Let $G$ be a simple algebraic group over an algebraically closed field $k$ of characteristic $p \geqs 0$ and let $\mathcal{C}_1, \ldots, \mathcal{C}_t$ be non-central conjugacy classes in $G$, where $t \geqs 2$ is an integer. Consider the irreducible subvariety $X = \mathcal{C}_1 \times \cdots \times \mathcal{C}_t$ of the Cartesian product $G^t$. Given a tuple $x = (x_1, \ldots, x_t) \in X$, let $G(x)$ denote the Zariski closure of the subgroup $\la x_1, \ldots, x_t\ra$ and set
\[
\Delta = \{ x \in X \, : \, G(x) = G\}.
\]
In this setting, a basic problem is to determine whether or not $\Delta$ is empty. Note that if $k$ is algebraic over a finite field then $G$ is locally finite and thus $\Delta$ is always empty, so this problem is only interesting when $k$ is not algebraic over a finite field, which is a hypothesis we adopt throughout the paper.

This problem has been the subject of several recent papers and some general results have been established. In characteristic zero, for example, a theorem of Guralnick \cite{gurnato} shows that $\Delta$ is always an open subset of $X$. In the general setting, \cite[Theorem 2]{BGG1} states that $\Delta$ is non-empty if and only if it is dense in $X$, while \cite[Theorem 1]{BGG2} reveals that $\Delta$ is non-empty if and only if it is generic, which means that it contains the complement of a countable union of proper closed subvarieties of $X$. If $G$ is a simple algebraic group of exceptional type, then \cite[Theorem 7]{BGG1} states that $\Delta$ is non-empty if $t \geqs 5$, and the same conclusion holds for $t \geqs 4$ when $G = G_2$. Moreover, this is best possible since there are examples where $t=4$ ($t=3$ for $G=G_2$) and $\Delta$ is empty (see \cite[Theorem 3.22]{BGG1}). 

In this paper, we extend the earlier work in \cite{BGG1} by studying the special case where $G$ is an exceptional type group in positive characteristic $p$ and each $\mathcal{C}_i$ is a conjugacy class of unipotent elements of order $p$ (note that every nontrivial unipotent element has order $p$ if $p \geqs h$, where $h$ is the Coxeter number of $G$). In this situation, our aim is to classify the varieties $X = \mathcal{C}_1 \times \cdots \times \mathcal{C}_t$ where $t \geqs 2$ and $\Delta$ is empty. This is in a similar spirit to the main theorem of \cite{BGG2}, which considers the analogous problem for symplectic and orthogonal groups when the $\mathcal{C}_i$ comprise elements of prime order modulo the centre of the group. In turn, this extends earlier work of Gerhardt \cite{Ger} on linear groups. Notice that all of these results are independent of the isogeny class of $G$ since the centre of $G$ is contained in the Frattini subgroup and thus a subgroup $H$ is dense in $G$ if and only if $HZ/Z$ is dense in $G/Z$, where $Z$ is any central subgroup of $G$. We will typically work with the simply connected form of the group.

Our main result is the following (Tables \ref{tab:main} and \ref{tab:special} are presented at the end of the paper in Section \ref{s:tab}). Note that any two involutions generate a dihedral group, which explains why we assume $p \geqs 3$ when $t=2$.

\begin{theorem}\label{t:main}
Let $G$ be a simple algebraic group of exceptional type over an algebraically closed field of characteristic $p>0$ that is not algebraic over a finite field. Set $X = \mathcal{C}_1 \times \cdots \times \mathcal{C}_t$, where $t \geqs 2$ and each $\mathcal{C}_i$ is a conjugacy class of elements of order $p$ in $G$. Assume $p \geqs 3$ if $t=2$. Then $\Delta$ is empty if and only if 
$X$ is one of the cases recorded in Tables \ref{tab:main} and \ref{tab:special}.
\end{theorem}

\begin{remk}\label{r:main}
Some remarks on the statement of Theorem \ref{t:main} are in order.
\begin{itemize}\addtolength{\itemsep}{0.2\baselineskip}
\item[{\rm (a)}] We adopt the notation for unipotent classes from \cite{LS_book} (see Tables 22.1.1-22.1.5 in \cite{LS_book}) and it is worth noting that this sometimes differs from the notation used by other authors. For example, the unipotent class in $E_7$ labelled $(A_1^3)^{(1)}$ in \cite{LS_book} is denoted $(3A_1)''$ in \cite{Law1,Spal}. Similarly, the class $E_6(a_3)$ in \cite{Law1,LS_book} is labelled $A_5+A_1$ in \cite{Spal}. 
\item[{\rm (b)}] By inspecting the relevant tables in \cite{Law1}, it is easy to read off the required condition on $p$ to ensure that each $\mathcal{C}_i$ contains elements of order $p$. To do this, we consider the action of $G$ on a suitable $kG$-module $V$ and we inspect the Jordan form of a representative $y \in \mathcal{C}_i$ in this representation, noting that $y$ has order $p$ if and only if every Jordan block has size at most $p$. For example, if $G = E_8$ and $y \in G$ is contained in the class $A_2$, then \cite[Table 9]{Law1} states that the Jordan form of $y$ on the adjoint module for $G$ is as follows
\[
\left\{\begin{array}{ll}
(J_4^2,J_3^{54},J_1^{78}) & p = 2 \\
(J_3^{57},J_1^{77}) & p = 3 \\
(J_5,J_3^{55},J_1^{78}) & p \geqs 5
\end{array}\right.
\]
where $J_i$ denotes a standard unipotent Jordan block of size $i$. Therefore, the elements in this class have order $p$ if and only if $p \geqs 3$ (they have order $4$ when $p=2$). 
\item[{\rm (c)}] We have chosen to record the cases $X = \mathcal{C}_1 \times \cdots \times \mathcal{C}_t$ with $\Delta$ empty over two tables, rather than one. This is essentially an artefact of our proof and it will be convenient to make a distinction between the cases in Tables \ref{tab:main} and \ref{tab:special} (for example, see Theorem \ref{t:fixV} below).
\item[{\rm (d)}] To avoid unnecessary repetition, the tuples in Tables \ref{tab:main} and \ref{tab:special} are listed up to reordering, and also up to graph automorphisms when $(G,p) = (G_2,3)$ or $(F_4,2)$. For example, if $(G,p) = (G_2,3)$ and $\tau$ is a graph automorphism of $G$, then $\tau$ interchanges the classes of long and short root elements (denoted by $A_1$ and $\tilde{A}_1$), whence $\Delta$ is also empty when $t=3$ and $(\mathcal{C}_1,\mathcal{C}_2,\mathcal{C}_3) = (\tilde{A}_1,\tilde{A}_1,\tilde{A}_1)$. 
\end{itemize}
\end{remk}

The Zariski closure of a unipotent conjugacy class is a union of unipotent classes and this leads naturally to a partial order on the set of unipotent classes of $G$, which has been completely determined by Spaltenstein (see \cite[Section II.10]{Spal} for $G = G_2$ and \cite[Section IV.2]{Spal} for the other types). This is relevant here because Proposition \ref{p:clos} (see \cite[Lemma 2.2]{BGG2}) states that $\Delta$ is empty if and only if $G(y) \ne G$ for all $y \in \bar{X}$, where $\bar{X}$ denotes the Zariski closure of $X$ in $G^t$. This observation allows us to present the following reformulation of Theorem \ref{t:main} (see Remark \ref{r:spal} for some brief comments on Spaltenstein's notation in \cite{Spal}).

\begin{corol}\label{c:main2}
The set $\Delta$ is empty if and only if $X$ is contained in the closure of one of the varieties $Y = \mathcal{C}_1' \times \cdots \times \mathcal{C}_t'$ recorded in Table \ref{tab:clos}, up to reordering and graph automorphisms if $(G,p) = (G_2,3)$ or $(F_4,2)$.
\end{corol}

{\small
\begin{table}
\[
\begin{array}{ll} \hline
G & (\mathcal{C}_1', \ldots, \mathcal{C}_t') \\ \hline
G_2 & (A_1,A_1,A_1), (A_1,G_2(a_1)) \\
F_4 & (A_1,A_1,A_1,A_1), (A_1,A_1,A_2), (A_1,\tilde{A}_1,\tilde{A}_1), (A_1,B_3), (\tilde{A}_1,\tilde{A}_2), (\tilde{A}_1, B_2), (A_2,A_2) \\
E_6 & (A_1,A_1,A_1,A_1), (A_1,A_1,A_2), (A_1,A_1^2,A_1^2), (A_1,D_4), (A_1,A_4), (A_1^2,A_3), (A_1^2,A_2^2), (A_2,A_2) \\
E_7 & (A_1,A_1,A_1,A_1), (A_1,A_1,A_2A_1), (A_1,(A_1^3)^{(1)},(A_1^3)^{(1)}),  (A_1,(A_5)^{(1)}), (A_1,D_5(a_1)) \\
&  ((A_1^3)^{(1)},(A_3A_1)^{(1)}), (A_2,A_2A_1) \\
E_8 & (A_1,A_1,A_1,A_1^2), (A_1,A_1,A_3), (A_1,A_1^2,A_2),  (A_1,D_5), (A_1,D_4A_2), (A_1^2,D_4), (A_2,A_3) \\
\hline
\end{array}
\]
\caption{The varieties $Y = \mathcal{C}_1' \times \cdots \times \mathcal{C}_t'$ in Corollary \ref{c:main2}}
\label{tab:clos}
\end{table}
}

Let $V$ be a $kG$-module with $C_V(G) = 0$ and write $\mathcal{C}_i = y_i^G$. Let $C_V(y_i)$ be the $1$-eigenspace of $y_i$ on $V$ and note that $\dim C_V(y_i)$ coincides with the number of Jordan blocks in the Jordan form of $y_i$ on $V$. If   
\begin{equation}\label{e:di}
\sum_{i=1}^t \dim C_V(y_i) > (t-1)\dim V
\end{equation}
then the intersection $\bigcap_i C_V(y_i)$ is nonzero and thus $G(x)$ has a nontrivial fixed space on $V$ for all $x \in X$. In particular, this condition implies that $\Delta$ is empty. 

With this observation in hand, our next result is an immediate corollary of Theorem \ref{t:main}. Here $V$ is the specific Weyl module for $G$ (or its dual) recorded in Table \ref{tab:mod}, where we label a set of fundamental dominant weights for $G$ in the usual way (see \cite{Bou}). Notice that $C_V(G)=0$ and the Jordan form of $y_i$ on $V$ has been determined by Lawther \cite{Law1}, which allows us to compute the sum in \eqref{e:di}. In part (ii) of the corollary, $\tau$ is a graph automorphism of $G$ and we write $\mathcal{C}_i^{\tau}$ for the image of the class $\mathcal{C}_i$ under $\tau$.

{\small
\begin{table}
\[
\begin{array}{cccccc} \hline
G & G_2 & F_4 & E_6 & E_7 & E_8 \\ \hline
V & W_G(\l_1)^* & W_G(\l_4)^* & W_G(\l_1) & W_G(\l_7) & W_G(\l_8) \\
\dim V & 7 & 26 & 27 & 56 & 248 \\ \hline
\end{array}
\]
\caption{The $kG$-module $V$ in Corollary \ref{c:main}} 
\label{tab:mod}
\end{table}}

\begin{corol}\label{c:main}
Let $V$ be the $kG$-module in Table \ref{tab:mod}. Then $\Delta$ is empty if and only if one of the following holds:
\begin{itemize}\addtolength{\itemsep}{0.2\baselineskip}
\item[{\rm (i)}] The inequality in \eqref{e:di} is satisfied.  
\item[{\rm (ii)}] $(G,p) = (G_2,3)$ or $(F_4,2)$, and \eqref{e:di} holds for
$\mathcal{C}_1^{\tau} \times \cdots \times \mathcal{C}_t^{\tau}$.
\item[{\rm (iii)}] $X$ is one of the cases listed in Table \ref{tab:special}, up to reordering and graph automorphisms if $(G,p) = (F_4,2)$.
\end{itemize} 
\end{corol}

Note that (ii) is needed here. For example, if $(G,p) = (G_2,3)$ then long and short root elements in $G$ have Jordan form $(J_2^2,J_1^3)$ and $(J_3,J_2^2)$ on $V$, respectively, so \eqref{e:di} holds for 
$(\mathcal{C}_1, \mathcal{C}_2, \mathcal{C}_3) = (A_1,A_1,A_1)$, but not for 
the image $(\tilde{A}_1,\tilde{A}_1,\tilde{A}_1)$ under $\tau$. Similarly, if $(G,p) = (F_4,2)$ then \eqref{e:di} holds for $(A_1,A_1,A_1,A_1)$, but not for 
$(\tilde{A}_1,\tilde{A}_1,\tilde{A}_1,\tilde{A}_1)$.

\begin{remk}\label{r:natural}
The previous corollary reveals that we can identify most cases where $\Delta$ is empty in Theorem \ref{t:main} just by considering the action of the relevant conjugacy class representatives on a suitable $kG$-module. Here we briefly recall that similar results for classical groups have been established in \cite{BGG2,Ger}, with respect to the natural module. 

Let $G$ be a simple algebraic group of classical type over an algebraically closed field $k$ of characteristic $p \geqs 0$ which is not algebraic over a finite field. Assume $p \ne 2$ if $G$ is of type $B_n$. Let $V$ be the natural module for $G$ and set $X = \mathcal{C}_1 \times \cdots \times \mathcal{C}_t$, where $t \geqs 2$ and each $\mathcal{C}_i = y_i^G$ is a non-central conjugacy class. Note that $C_V(G) = 0$, which means that $\Delta$ is empty if \eqref{e:di} holds. In \cite{Ger}, Gerhardt proves that, if $G={\rm SL}(V)$ and $\dim V \geqs 3$, then $\Delta$ is empty if and only if either \eqref{e:di} holds, or $t=2$ and both $y_1$ and $y_2$ have quadratic minimal polynomials on $V$ (if the latter property holds, then for all $x \in X$, every composition factor of $G(x)$ on $V$ is at most $2$-dimensional). A similar result is proved for the groups ${\rm Sp}(V)$ and ${\rm SO}(V)$ in \cite{BGG2} under the assumption that the $y_i$ have prime order modulo $Z(G)$. The analysis of the latter groups is more complicated and several exceptions arise where $\Delta$ is empty and \eqref{e:di} does not hold for the natural module $V$ (see \cite[Tables 1,2]{BGG2}).
\end{remk}

We have noted that $\Delta$ is empty if the inequality in \eqref{e:di} is satisfied for a suitable $kG$-module $V$ with $C_V(G)=0$. In order to study the general case, let $\mathcal{M}$ be a complete set of representatives of the conjugacy classes of closed positive dimensional maximal subgroups of $G$. Then $\mathcal{M}$ is finite and the classification of the subgroups in $\mathcal{M}$ was completed by Liebeck and Seitz in \cite{LS04} (modulo the existence of an additional subgroup $H = F_4$ in $\mathcal{M}$, which arises when $(G,p) = (E_8,3)$; see \cite{CST}). Following \cite{BGG1}, set 
\[
\Delta^+ = \{x \in X \,:\, \dim G(x) > 0 \}
\]
and note that $\Delta = \Delta^+ \cap \Lambda$, where $\Lambda$ is the set of $x \in X$ such that $G(x)$ is not contained in a positive dimensional maximal subgroup of $G$. In addition, we define
\begin{equation}\label{e:XHH}
X_H = \{ x \in X \,:\, \mbox{$G(x) \leqs H^g$ for some $g \in G$} \}
\end{equation}
for an arbitrary closed subgroup $H$ of $G$.
 
By combining \cite[Corollary 4]{BGG1} and \cite[Theorem 1.1]{GMT}, we see that $\Delta^{+}$ is empty if and only if $(G,p) = (G_2,3)$ or $(F_4,2)$, with $t=2$ and $\{\mathcal{C}_1,\mathcal{C}_2\} = \{A_1,\tilde{A}_1\}$, so we may assume $\Delta^+$ is non-empty. Then $\Delta^+$ is a dense subset of $X$ by \cite[Theorem 1]{BGG1} and thus $\Delta$ is non-empty if $\Lambda$ contains a non-empty open subset of $X$. As explained in \cite{BGG1}, we can work with fixed point spaces in order to study the existence (or otherwise) of such a subset of $\Lambda$. 

For $H \in \mathcal{M}$, let $\O$ be the coset variety $G/H$ and let 
$C_{\O}(y) = \{ \a \in \O \,:\, \a^{y} = \a\}$
be the fixed point space of an element $y \in G$ on $\O$. Set 
\[
\a(G,H,y) = \frac{\dim C_{\O}(y)}{\dim \O}
\]
and write 
\[
\Sigma_X(H) = \sum_{i=1}^t \a(G,H,y_i),
\]
where $X = \mathcal{C}_1 \times \cdots \times \mathcal{C}_t$ and $\mathcal{C}_i = y_i^G$ as before. As noted in the proof of \cite[Theorem 5]{BGG1}, if 
\begin{equation}\label{e:fix}
\Sigma_X(H) < t-1
\end{equation}
then $X_H$ is contained in a proper closed subset of $X$. In particular, if the inequality in \eqref{e:fix} holds for all $H \in \mathcal{M}$, then $\Lambda$ contains a non-empty open subset (recall that $\mathcal{M}$ is a finite set) and thus $\Delta$ is non-empty. 

In order to establish the inequality in \eqref{e:fix} for a given subgroup $H \in \mathcal{M}$, we need upper bounds on $\dim C_{\O}(y)$ for elements $y \in G$ of order $p$, where $\O = G/H$. Fixed point spaces for actions of exceptional algebraic groups were studied by Lawther, Liebeck and Seitz \cite{LLS1} in a general setting, but we will often require sharper bounds on $\dim C_{\O}(y)$ for the unipotent elements we are interested in here. To do this, first observe that $\dim C_{\O}(y)=0$ if $H$ does not contain a conjugate of $y$, so we may as well assume $y \in H$. Then \cite[Proposition 1.14]{LLS1} states that
\[
\dim C_{\O}(y) = \dim \O - \dim y^G + \dim (y^G \cap H).
\]
In order to apply this formula, given an element $y \in H$ of order $p$, we need to identify the $G$-class of $y$ and we need to compute $\dim(y^G \cap H)$. In many cases, we can appeal to earlier work of Lawther to do this. Specifically, if $M$ is a maximal closed connected reductive subgroup of $G$, then the $G$-class of each unipotent $M$-class is determined in \cite{Law2} and we will make extensive use of these results. Indeed, this essentially reduces our problem to the following two cases:
\begin{itemize}\addtolength{\itemsep}{0.2\baselineskip}
\item[{\rm (a)}] $H \in \mathcal{M}$ is reductive and $|H:H^0|$ is divisible by $p$;
\item[{\rm (b)}] $H \in \mathcal{M}$ is a parabolic subgroup. 
\end{itemize}

In (a) we need to consider the possible existence of elements $y \in H \setminus H^0$ of order $p$. In this situation, we will often work with the restriction of a suitable $kG$-module $W$ to $H^0$ in order to determine the Jordan form on $W$ of such an element $y$. From here, in almost all cases, we can then identify the $G$-class of $y$ by inspecting \cite{Law1}. 

The case where $H$ is a maximal parabolic subgroup requires special attention. Here we adopt an indirect approach, which involves working with the corresponding permutation character $1^{G_{\s}}_{H_{\s}}$ in order to compute $\dim C_{\O}(y)$, where $\s$ is a Steinberg endomorphism of $G$ and $G_{\s} = G(q)$ for some $p$-power $q$. In turn, this relies heavily on work of Geck and L\"{u}beck \cite{Geck,Lub}, who have very recently completed the computation of the Green functions for finite exceptional groups of Lie type in all characteristics (see Section \ref{s:parab} for more details).

For most varieties $X = \mathcal{C}_1 \times \cdots \times \mathcal{C}_t$ as in Theorem \ref{t:main}, we will show that either \eqref{e:di} is satisfied for the $kG$-module $V$ in Table \ref{tab:mod}, in which case $\Delta$ is empty, or the inequality in \eqref{e:fix} holds for all $H \in \mathcal{M}$ and thus $\Delta$ is non-empty. In this way, the proof of Theorem \ref{t:main} is reduced to the configurations in Table \ref{tab:special}, where in each case the inequality in \eqref{e:di} is not satisfied and $\Sigma_X(H) \geqs t-1$ for some $H \in \mathcal{M}$. We need a different approach to show that $\Delta$ is empty in these special cases. To do this, we formalise a technique which was first introduced in \cite{BGG2} (see the proof of \cite[Lemma 4.1]{BGG2}, for example), which provides an essentially uniform method for handling all of these cases. 

The basic idea is as follows. First we embed a set of class representatives $y_i \in \mathcal{C}_i$ in a carefully chosen closed connected proper subgroup $L$ of $G$ and we then consider the morphism
\[
\varphi: \mathcal{D}_1 \times \cdots \times \mathcal{D}_t \times G \to X, \;\; (a_1,\ldots, a_t,g) \mapsto (a_1^g, \ldots, a_t^g),
\]
where $\mathcal{D}_i = y_i^L$. By studying the fibres of this map, we aim to show that $\varphi$ is a dominant morphism and that every tuple in the image of $\varphi$ topologically generates a subgroup contained in a conjugate of $L$. Then the set $X_L$ defined in \eqref{e:XHH} contains a non-empty open subset of $X$, so $\Delta$ is not dense in $X$ and thus \cite[Theorem 2]{BGG1} implies that $\Delta$ is empty. We refer the reader to Proposition \ref{p:fibre} for more details.

\begin{remk}\label{r:prime}
In Theorem \ref{t:main} we assume that each unipotent class $\mathcal{C}_i = y_i^G$ contains elements of order $p$. To conclude this introductory section, we briefly discuss the more general problem, where the $\mathcal{C}_i$ are arbitrary nontrivial unipotent classes.
\begin{itemize}\addtolength{\itemsep}{0.2\baselineskip}
\item[{\rm (a)}] As noted above, every nontrivial unipotent element has order $p$ if 
$p \geqs h$, where $h$ is the Coxeter number of $G$, so Theorem \ref{t:main} gives a complete solution to the problem under this hypothesis. For the reader's convenience, we recall that $h = 6, 12, 12, 18, 30$ for $G = G_2, F_4, E_6, E_7, E_8$, respectively.
\item[{\rm (b)}] It is straightforward to show that $\Delta$ is empty for every tuple in Table \ref{tab:main} for all $p \geqs 0$. Indeed, if $V$ is the $kG$-module in Table \ref{tab:mod} then by inspecting \cite{Law1} one checks that \eqref{e:di} holds and thus $\bigcap_i C_V(y_i) \ne 0$ in each case. Similarly, \eqref{e:di} holds in each of the following cases, where $t = p = 2$, $\mathcal{C}_1$ is a class of involutions and $\mathcal{C}_2$ contains elements of order $4$:

\vspace{1mm}

\begin{itemize}\addtolength{\itemsep}{0.2\baselineskip}
\item $G = F_4$: $(A_1,(B_2)_2)$, $(A_1,(\tilde{A}_2A_1)_2)$, $(A_1,(C_3(a_1))_2)$, $((\tilde{A}_1)_2, A_2)$
\item $G = E_7$ or $E_8$: $(A_1,(A_3A_2)_2)$
\end{itemize}
\item[{\rm (c)}] Suppose the characteristic $p$ is a good prime for $G$ in the usual sense, so $p \geqs 5$, with $p \geqs 7$ for $G = E_8$. Then one can check that every class $\mathcal{C}_i$ appearing in one of the tuples in Table \ref{tab:special} contains elements of order $p$ (and thus $\Delta$ is empty by Theorem \ref{t:main}), with the exception of the following two cases (up to reordering):

\vspace{1mm}

\begin{itemize}\addtolength{\itemsep}{0.2\baselineskip}
\item $G = E_7$, $p=5$ and $(\mathcal{C}_1,\mathcal{C}_2) = (A_1,(A_5)^{(1)})$
\item $G = E_8$, $p=7$ and $(\mathcal{C}_1,\mathcal{C}_2) = (A_1,D_5)$
\end{itemize}

\vspace{1mm}

\noindent Without some additional work, we cannot resolve these cases by appealing to Proposition \ref{p:fibre}, due to the prime order hypothesis adopted in \cite[Theorem 7]{BGG2} (for the group $L=D_6$) and in the present paper (for $L = E_7$). 

\item[{\rm (d)}] We claim that $\Delta$ is empty in the two special cases highlighted in part (c), which allows us to conclude that $\Delta$ is empty for every tuple in Table \ref{tab:special} when $p$ is a good prime for $G$. We thank Bob Guralnick for suggesting the following argument. 

Let $x = (x_1,x_2) \in X$ and note that $G(x) \ne G$ if $G(x)$ acts reducibly on the adjoint module for $G$. Write $k = R/M$, where $R$ is a ring of algebraic integers and $M$ is a maximal ideal, and lift  $x$ to $z = (z_1,z_2) \in \mathcal{C}_1' \times \mathcal{C}_2'$, where $\mathcal{C}_1'$ and $\mathcal{C}_2'$ are the corresponding unipotent conjugacy classes in $G(R)$, with the same labels as $\mathcal{C}_1$ and $\mathcal{C}_2$. By a standard compactness argument, which relies on the fact that $\Delta$ is empty in the two cases of interest when the characteristic of the underlying field is sufficiently large, we deduce that $G(R)$ is not topologically generated by $z_1$ and $z_2$. In turn, this implies that $\la z_1,z_2\ra$ acts reducibly on the adjoint module for $G(R)$ and hence $\la x_1, x_2\ra$ is reducible on the adjoint module for $G$. Therefore, $G(x) \ne G$ and we conclude that $\Delta$ is empty.

Alternatively, if $G = E_7$ and $(\mathcal{C}_1,\mathcal{C}_2) = (A_1,(A_5)^{(1)})$ then we can use \cite[Corollary 3.20]{BGG2} to show that $\Delta$ is empty for all $p \geqs 0$ (see Remark \ref{r:e7_lie}).

\item[{\rm (e)}] We conjecture that the conclusion to Theorem \ref{t:main} holds in good characteristic, regardless of the orders of the elements in each unipotent class.
\end{itemize}
\end{remk}

\vs

\noindent \textbf{Acknowledgements.} I am indebted to two anonymous referees for their very helpful comments and suggestions on an earlier version of this paper. I also thank Bob Guralnick and Donna Testerman for helpful discussions on the content of this paper, and I am grateful for the generous hospitality of the Institute of Mathematics at the \'{E}cole Polytechnique F\'{e}d\'{e}rale de Lausanne during a research visit in 2022. I also acknowledge the support of the Swiss National Science Foundation, research grant number 207730.

\section{Preliminaries}\label{s:prel}

Here we record some preliminary results which will be useful in the proof of Theorem \ref{t:main}. Throughout this section, $G$ denotes a simple algebraic group over an algebraically closed field $k$ of characteristic $p \geqs 0$.

\subsection{Fixed point spaces and conjugacy classes}\label{ss:fix}

Let $H$ be a closed subgroup of $G$ and consider the natural action of $G$ on the coset variety $\O = G/H$. For $y \in G$ we define
\[
C_{\O}(y) = \{ \a \in \O \,:\, \a^y = \a\}.
\]
As briefly explained in Section \ref{s:intro}, bounds on the dimensions of these fixed point spaces will play a key role in the proof of Theorem \ref{t:main}. In order to obtain such bounds, we will repeatedly apply the following result (see  \cite[Proposition 1.14]{LLS1}). 

\begin{prop}\label{p:dim}
We have 
\[
\dim C_{\O}(y) = \dim \O - \dim y^G + \dim(y^G \cap H)
\] 
for all $y \in H$.
\end{prop}

Here $\dim \O = \dim G - \dim H$ and it is also easy to compute $\dim y^G = \dim G - \dim C_G(y)$ if we can identify the $G$-class of $y$. However, the final term $\dim(y^G\cap H)$ is more difficult to calculate, in general, since one needs to understand the embedding of $H$-classes in $G$. In the main setting we are interested in here, with $G$ of exceptional type and $y \in G$ unipotent, we will work extensively with Lawther's results in \cite{Law2} on the fusion of unipotent $H$-classes in $G$ when $H$ is a maximal closed connected reductive subgroup of $G$. Note that if $H$ is reductive (and possibly disconnected) then $y^G \cap H = \bigcup_i y_i^H$ is a finite union of $H$-classes for all $y \in H$ (see \cite{Gur}) and thus 
\[
\dim(y^G \cap H) = \max_i \dim y_i^H
\] 
in this situation.

In order to prove Theorem \ref{t:main} we need detailed information on the unipotent classes in simple algebraic groups and there is an extensive literature to draw upon. First assume $G$ is an exceptional group. Here we refer the reader to \cite{LS_book} and specifically Tables 22.1.1-5, where the relevant conjugacy classes are listed. Throughout this paper, we will adopt the labelling of unipotent classes given in these tables, noting that this sometimes differs from the notation used elsewhere (for example, see Remark \ref{r:main}(a)). The centraliser dimensions are recorded in the third column of each table, so it is easy to compute $\dim y^G$ for each unipotent element $y \in G$. We will also make extensive use of Lawther's work in \cite{Law1}, where he determines the Jordan form of each unipotent element on certain $kG$-modules $V$, including the module defined in Table \ref{tab:mod}. In particular, we can use \cite{Law1} to read off the unipotent classes containing elements of order $p$, recalling that $y$ has order $p$ if and only if every Jordan block in the Jordan form of $y$ on $V$ has size at most $p$.

We will also need some results from \cite{LS_book} on unipotent classes in classical type algebraic groups. Let $G$ be such a group, with natural module $V$. As noted in \cite[Theorem 3.1]{LS_book}, if $p \ne 2$ then each unipotent class $y^G$ is essentially determined by the Jordan form of $y$ on $V$, which also encodes the dimension of the class. The description of unipotent classes is more complicated when $p=2$ and $G$ is a symplectic or orthogonal group (see \cite[Theorem 4.2]{LS_book}). Here the conjugacy classes of involutions were originally determined by Aschbacher and Seitz \cite[Sections 7,8]{AS} and we will often use their notation for class representatives. 

The Zariski closure of a unipotent conjugacy class $y^G$ is a union of unipotent classes and this yields a natural partial order on the set of unipotent classes of $G$, where we write $z^G \preccurlyeq y^G$ if $z^G$ is contained in the closure of $y^G$. In view of Proposition \ref{p:clos} below, we are interested in this closure relation when $G$ is an exceptional type group and here we can appeal to work of Spaltenstein \cite{Spal}, which gives a complete description of the corresponding ordering in this setting. For $G = G_2$, we refer the reader to \cite[Section II.10]{Spal}, while the relevant closure diagrams for the remaining exceptional groups are presented in \cite[Section IV.2]{Spal}.

\begin{rem}\label{r:spal}
As previously remarked, our labelling of unipotent classes (which is consistent with \cite{LS_book}) does not always agree with the notation adopted by Spaltenstein in \cite{Spal}. This sometimes means that extra care is required when interpreting Spaltenstein's closure diagrams. For example, suppose $G = F_4$ and $p=2$. Here \cite{LS_book} uses the notation $\tilde{A}_1$ for the class of short root elements in $G$ and $(\tilde{A}_1)_2$ for the special class of involutions with dimension $22$. But this is opposite to the notation in \cite{Spal}, where $(\tilde{A}_1)_2$ denotes the class of short root elements. In particular, Spaltenstein's closure diagram (see \cite[p.250]{Spal}) shows that the closure of every nontrivial unipotent class contains long or short root elements, which is consistent with \cite[Corollary 3.3]{GMal}.
\end{rem}

We will also need the following version of \cite[Proposition 1.4]{LLS1} on graph automorphisms of order $p$. In Table \ref{tab:graph}, we write $C_H(u)$ for the centraliser of a long root element $u \in H$. 

\begin{prop}\label{p:graph}
Let $G$ be a simple algebraic group of type $A_r$, $E_6$ or $D_4$ over an algebraically closed field of characteristic $p$, where $p=2,2$ or $3$, respectively. Let $\tau$ be a graph automorphism of $G$ of order $p$ and let $y_1, \ldots, y_n$ be a complete set of representatives of the $G$-classes of elements of order $p$ in the coset $G\tau$. Then $n$ and $d_i = \dim y_i^G$ are recorded in Table \ref{tab:graph}, together with the structure of $C_G(y_i)$.
\end{prop}

{\small
\begin{table}
\[
\begin{array}{lcccl} \hline
G & p & n & d_i & C_G(y_i) \\ \hline
A_{2m} & 2 & 1 & m(2m+3) & B_m  \\
A_{2m-1} & 2 & 2 & 2m^2 \mp m-1 & C_m, \, C_{C_m}(u) \\ 
E_6 & 2 & 2 & 26, 42 & F_4, \, C_{F_4}(u)  \\
D_4 & 3 & 2 & 14, 20 & G_2, \, C_{G_2}(u)  \\ \hline
\end{array}
\]
\caption{The classes and centralisers of graph automorphisms of order $p$}
\label{tab:graph}
\end{table}}

\subsection{Subgroup structure}\label{ss:sub}

Let $G$ be a simple algebraic group of exceptional type and let $\mathcal{M}$ be a set of representatives of the conjugacy classes of closed positive dimensional maximal subgroups of $G$. We may write 
\[
\mathcal{M} = \mathcal{P} \cup \mathcal{R},
\]
where the subgroups in $\mathcal{P}$ are parabolic and those in $\mathcal{R}$ are  reductive (and possibly disconnected). Recall that the conjugacy classes of maximal parabolic subgroups are parameterised by the nodes in the Dynkin diagram of $G$; we will use $P_i$ for a maximal parabolic subgroup corresponding to the $i$-th node in the Dynkin diagram (throughout this paper, we adopt the standard labelling of Dynkin diagrams as in \cite{Bou}). The classification of the subgroups in $\mathcal{R}$ was completed by Liebeck and Seitz in \cite{LS04}, modulo the existence of an additional conjugacy class of maximal subgroups of type $F_4$ when $(G,p) = (E_8,3)$; see \cite{CST}. The subgroups in $\mathcal{R}$ are recorded in Table \ref{tab:max}. 

\begin{rem}
In Table \ref{tab:max} we use the notation $\tilde{L}$ to denote a simple factor $L$ of $H^0$ which is generated by short root subgroups of $G$. In addition, if $G = E_r$ then $T$ is a maximal torus of $G$ and $W = N_G(T)/T$ is the Weyl group, where $W$ is isomorphic to ${\rm PGSp}_4(3) = {\rm PSp}_4(3).2$, $2 \times {\rm Sp}_6(2)$ and $2.{\rm O}_{8}^{+}(2) = 2.\O_{8}^{+}(2).2$ for $r = 6,7$ and $8$, respectively.
\end{rem}

\renewcommand{\arraystretch}{1.2}
{\small
\begin{table}
\[
\begin{array}{ll} \hline
G & \mathcal{R} \\ \hline
G_2 & A_2.2, \, \tilde{A}_2.2 \, (p=3), \, A_1\tilde{A}_1, \, A_1 \, (p \geqs 7) \\
F_4 & B_4, \, C_4 \, (p=2), \, A_1C_3 \, (p \geqs 3), \, A_1G_2 \, (p \geqs 3), \, G_2 \, (p=7), \, A_1 \, (p \geqs 13), \, D_4.S_3, \, \tilde{D}_4.S_3 \, (p=2), \, A_2\tilde{A}_2.2 \\ 
E_6 & F_4, \, A_1A_5, \, C_4 \, (p \geqs 3), \, A_2G_2, \, G_2 \, (p \ne 7), \, D_4T_2.S_3,\, A_2^3.S_3, \, A_2.2 \, (p \geqs 5), \, T.W \\
E_7 & A_1D_6, \, A_1F_4, \, G_2C_3, \, A_1G_2 \, (p \geqs 3), \, A_1A_1 \, (p \geqs 5), \,   A_1\, (\mbox{$2$ classes; $p \geqs 17,19$}), \, E_6T_1.2, \, A_7.2 \\ 
& A_2A_5.2,\, A_1^3D_4.S_3, \, (2^2 \times D_4).S_3 \, (p \geqs 3), \, A_1^7.{\rm L}_3(2), \, A_2.2 \, (p \geqs 5), \, T.W \\ 
E_8 & A_1E_7, \, A_2E_6.2, \, D_8, \, G_2F_4, \, F_4 \, (p=3), \, B_2 \, (p \geqs 5), \, A_1 \, (\mbox{$3$ classes, $p \geqs 23,29,31$}), \, A_8.2, \, D_4^2.(S_3 \times 2) \\
& A_4^2.4, \, A_1G_2^2.2 \, (p \geqs 3), \,  A_2^4.{\rm GL}_2(3), \, A_1^8.{\rm AGL}_3(2), \, A_2A_1.2 \, (p \geqs 5), \, A_1 \times S_5 \, (p \geqs 7),\, T.W  \\ \hline
\end{array}
\]
\caption{The collection $\mathcal{R}$ of reductive maximal subgroups of $G$}
\label{tab:max}
\end{table}}
\renewcommand{\arraystretch}{1}

The following result on the maximal overgroups of long root elements will be useful later. Recall that the long root elements comprise the class labelled $A_1$ in \cite[Tables 22.2.1-5]{LS_book}.

\begin{thm}\label{t:long}
Let $H \in \mathcal{R}$. Then $H$ contains a long root element of $G$ only if $H \in \mathcal{L}$, where $\mathcal{L}$ is defined in Table \ref{tab:long}.
\end{thm}

\begin{proof}
We need to rule out the existence of long root elements in each subgroup $H \in \mathcal{R} \setminus \mathcal{L}$ and in most cases we can use \cite{Law2} to do this. For example, suppose $G = E_8$. If $H = F_4$ and $p=3$ then the $G$-class of each unipotent $H$-class is presented in \cite[Table 2]{CST} and we immediately deduce that $H$ does not contain any long root elements. By inspecting \cite[Tables 36, 37]{Law2}, we see that the same conclusion holds when $H = A_2A_1.2$ or $B_2$ (both with $p \geqs 5$), and for $H = A_1$ we can appeal to \cite[Table 27]{Law2}. Finally, suppose $p \geqs 7$ and $H = A_1 \times S_5$. If $y \in H$ has order $p$, then we may embed $y$ in a connected maximal rank subgroup $A_4^2$ of $G$ so that $y = y_1y_2$ and each $y_i \in A_4$ is regular (see \cite[Lemma 1.5]{LS90}, for example). By inspecting \cite[Table 26]{Law2} we deduce that $y$ is contained in the $G$-class $E_8(a_7)$ and thus $H$ does not contain any long root elements.  

The other groups can be handled in the same way and we omit the details. Note that if $G = E_7$ and $H = (2^2 \times D_4).S_3$ (with $p \geqs 3$) then the proof of \cite[Proposition 5.12]{BGS} shows that $H$ does not contain long root elements. 
\end{proof}

\renewcommand{\arraystretch}{1.2}
{\small
\begin{table}
\[
\begin{array}{ll} \hline
G & \mathcal{L} \\ \hline
G_2 & A_2.2, \,  A_1\tilde{A}_1 \\
F_4 & B_4, \, C_4 \, (p=2), \, A_1C_3 \, (p \geqs 3), \, A_1G_2 \, (p \geqs 3), \, D_4.S_3, \, \tilde{D}_4.S_3 \, (p=2), \, A_2\tilde{A}_2.2 \\ 
E_6 & F_4, \, A_1A_5, \, C_4 \, (p \geqs 3), \, A_2G_2, \, D_4T_2.S_3,\, A_2^3.S_3, \, T.W \\
E_7 & A_1D_6, \, A_1F_4, \, G_2C_3, \, E_6T_1.2, \, A_7.2,\, A_2A_5.2,\, A_1^3D_4.S_3, \, A_1^7.{\rm L}_3(2), \, T.W \\ 
E_8 & A_1E_7, \, A_2E_6.2, \, D_8, \, G_2F_4, \, A_8.2, \, D_4^2.(S_3 \times 2),\, A_4^2.4, \, A_1G_2^2.2 \, (p \geqs 3), \,  A_2^4.{\rm GL}_2(3), \, A_1^8.{\rm AGL}_3(2),\, T.W  \\ \hline
\end{array}
\]
\caption{The subgroup collection $\mathcal{L}$ in Theorem \ref{t:long}}
\label{tab:long}
\end{table}}
\renewcommand{\arraystretch}{1}
 
\subsection{Topological generation}\label{ss:top}

For the remainder of Section \ref{s:prel}, let us adopt some of the notation introduced in Section \ref{s:intro}. Let $G$ be a simple algebraic group defined over an algebraically closed field $k$ of characteristic $p \geqs 0$ and let us assume $k$ is not algebraic over a finite field. Let $t \geqs 2$ and set $X = \mathcal{C}_1 \times \cdots \times \mathcal{C}_t$, where each $\mathcal{C}_i = y_i^G$ is a non-central conjugacy class. Given a tuple $x = (x_1, \ldots, x_t) \in X$, let $G(x)$ be the Zariski closure of $\la x_1, \ldots, x_t \ra$ and set
\begin{align*}
\Delta & = \{x \in X \,:\, G(x) = G\} \\
\Delta^{+} & = \{x \in X \,:\, \dim G(x) > 0\}. 
\end{align*}
In addition, let $\Lambda$ be the set of elements $x \in X$ such that $G(x)$ is not contained in a positive dimensional maximal subgroup of $G$. Note that $\Delta = \Delta^+ \cap \Lambda$. 

The following basic observation will be very useful (see \cite[Lemma 2.2]{BGG2}) and it explains why we are interested in the closure relation on unipotent classes discussed in Section \ref{ss:fix}.

\begin{prop}\label{p:clos}
Let $\bar{X}$ be the Zariski closure of $X$ in $G^t$ and assume $G = G(y)$ for some $y \in \bar{X}$. Then $\Delta$ is non-empty.
\end{prop}

We will also need the following result (see \cite[Theorems 1 and 2]{BGG1}). 

\begin{thm}\label{t:plus}
Let $\Gamma = \Delta$ or $\Delta^+$. Then $\Gamma$ is non-empty if and only if it is dense in $X$.
\end{thm}

For $H \in \mathcal{M}$, let $\O$ be the coset variety $G/H$ and set 
\begin{equation}\label{e:alpha}
\a(G,H,y) = \frac{\dim C_{\O}(y)}{\dim \O}
\end{equation}
for $y \in G$. Then define 
\begin{equation}\label{e:sig}
\Sigma_X(H) = \sum_{i=1}^t \a(G,H,y_i)
\end{equation}
where $\mathcal{C}_i = y_i^G$ for each $i$. The following result, which is essentially \cite[Theorem 5]{BGG1}, will be a key tool in the proof of Theorem \ref{t:main}. In particular, it establishes a bridge between topological generation and the dimensions of fixed point spaces on coset varieties.

\begin{prop}\label{p:fix}
If $\Delta^+$ is non-empty and $\Sigma_X(H)<t-1$ for all $H \in \mathcal{M}$, then $\Delta$ is non-empty.
\end{prop}

\begin{proof}
Fix a subgroup $H \in \mathcal{M}$. Since $\Sigma_X(H)<t-1$, the proof of \cite[Theorem 5]{BGG1} shows that
\[
X_H = \{ x \in X \,:\, \mbox{$G(x) \leqs H^g$ for some $g \in G$}\}
\]
is contained in a proper closed subset of $X$. Therefore, there is a non-empty open subset $U_H$ of $X$ such that for all $x \in U_H$, $G(x)$ is not contained in any conjugate of $H$. Since $\mathcal{M}$ is finite and the inequality $\Sigma_X(H)<t-1$ holds for all $H \in \mathcal{M}$, it follows that 
$\bigcap_{H \in \mathcal{M}} U_H$ is a non-empty open subset of $X$ contained in $\Lambda$. Finally, we recall that $\Delta^+$ is dense in $X$ by Theorem \ref{t:plus} and thus $\Delta = \Delta^+ \cap \Lambda$ is non-empty.
\end{proof}

In order to apply Proposition \ref{p:fix} in the proof of Theorem \ref{t:main}, we will need the following result. This follows by combining  \cite[Corollary 4]{BGG1} and \cite[Theorem 1.1]{GMT}.

\begin{thm}\label{t:plus2}
Let $G$ be a simple algebraic group of exceptional type and assume $X = \mathcal{C}_1 \times \cdots \times \mathcal{C}_t$, where each $\mathcal{C}_i$ is a nontrivial unipotent class. Then $\Delta^+$ is empty if and only if $(G,p) = (G_2,3)$ or $(F_4,2)$, with $t=2$ and $\{\mathcal{C}_1,\mathcal{C}_2\} = \{A_1,\tilde{A}_1\}$.
\end{thm}

To conclude this preliminary section, we state and prove the following result. This is based on a method first introduced in the proof of \cite[Lemma 4.1]{BGG2} and it will turn out to be a very useful tool in the proof of Theorem \ref{t:main}. Indeed, we will use it to show that $\Delta$ is empty for every case recorded in Table \ref{tab:special}. This result applies in the general setting, where $X = \mathcal{C}_1 \times \cdots \times \mathcal{C}_t$ and the $\mathcal{C}_i$ are arbitrary conjugacy classes.

\begin{prop}\label{p:fibre}
Suppose there exists a closed connected subgroup $L$ of $G$ and a tuple $x = (x_1, \ldots, x_t) \in X$ such that all of the following conditions are satisfied:
\begin{itemize}\addtolength{\itemsep}{0.2\baselineskip}
\item[{\rm (i)}] $M = N_G(L)$ is connected and $x_i \in M$ for all $i$.
\item[{\rm (ii)}] $G(x)^0 = L$ and $G(y)^0 \leqs L$ for all $y \in Y$, where $Y = \mathcal{D}_1 \times \cdots \times \mathcal{D}_t$ and $\mathcal{D}_i = x_i^M$.
\item[{\rm (iii)}] $\dim X - \dim Y = \dim G - \dim M$.
\end{itemize}
Then $\Delta$ is empty.
\end{prop}

\begin{proof}
Consider the morphism
\[
\varphi: Y \times G \to X, \;\; (d_1,\ldots,d_t,g) \mapsto (d_1^g, \ldots, d_t^g)
\]
and set $Z = \{ (x_1^{a^{-1}}, \ldots, x_t^{a^{-1}},a) \,:\, a \in M\}$. Then $Z$ is contained in the fibre $\varphi^{-1}(x)$ and we claim that $Z = \varphi^{-1}(x)$. To see this, let $(x_1^{g^{-1}}, \ldots, x_t^{g^{-1}},g) \in \varphi^{-1}(x)$ be an arbitrary element, so $g \in G$ and for each $i$ we may write $x_i^{g^{-1}} = x_i^{a_i}$ for some $a_i \in M$. Set $y = (x_1^{a_1}, \ldots, x_t^{a_t}) \in Y$. Then 
\[
\la x_1, \ldots, x_t \ra = \la x_1^{a_1}, \ldots, x_t^{a_t} \ra^g
\]
and using the conditions in (ii) we deduce that
\[
L = G(x)^0 = (G(y)^0)^g \leqs L^g.
\]
Since $L$ is connected, it follows that $L = L^g$ and thus $g \in N_G(L)$, which is equal to $M$ by (i). This justifies the claim and we conclude that $\dim \varphi^{-1}(x) = \dim M$. 

Next observe that $\varphi$ is a morphism of irreducible varieties, so we have
\[
\dim \varphi^{-1}(x) \geqs \dim (Y \times G) - \dim \overline{{\rm im}(\varphi)} \geqs \dim (Y \times G) - \dim X = \dim M,
\]
where the final equality holds by (iii). But we have already shown that $\dim \varphi^{-1}(x) = \dim M$, so this implies that $\dim \overline{{\rm im}(\varphi)} = \dim X$ and thus $\varphi$ is dominant. As a consequence, ${\rm im}(\varphi)$ contains a non-empty open subset $U$ of $X$. 

Finally, if $z = (d_1^g, \ldots, d_t^g) \in {\rm im}(\varphi)$, then 
$G(z)^0 = (G(y)^0)^g \leqs L^g$ with $y = (d_1, \ldots, d_t) \in Y$ and thus $G(z) \ne G$. It follows that $U \cap \Delta$ is empty, whence $\Delta$ is non-dense in $X$ and therefore empty by Theorem \ref{t:plus}.
\end{proof}

\section{Fixed spaces}\label{s:fixV}

Let $G$ be a simple algebraic group of exceptional type over an algebraically closed field $k$ of characteristic $p>0$ and assume $k$ is not algebraic over a finite field. Fix an integer $t \geqs 2$ and set $X = \mathcal{C}_1 \times \cdots \times \mathcal{C}_t$, where each $\mathcal{C}_i = y_i^G$ is a unipotent conjugacy class of elements of order $p$. Let $V$ be the $kG$-module defined in Table \ref{tab:mod} and let $C_V(y_i)$ be the fixed space of $y_i$ on $V$, in which case $\dim C_V(y_i)$ coincides with the number of Jordan blocks in the Jordan form of $y_i$ on $V$. Note that 
\[
C_V(G) = \{ v \in V \,:\, \mbox{$v^g = v$ for all $g \in G$} \} = 0.
\]

The purpose of this section is to record the following result. Since the Jordan form on $V$ of every unipotent element in $G$ is recorded in \cite{Law1}, the proof is a routine exercise.

\begin{thm}\label{t:fixV}
Let $V$ be the $kG$-module in Table \ref{tab:mod}. Then 
\[
\sum_{i=1}^t \dim C_V(y_i) > (t-1)\dim V
\]
if and only if $(\mathcal{C}_1, \ldots, \mathcal{C}_t)$ is one of the cases recorded in Table \ref{tab:main}, up to reordering and graph automorphisms if $(G,p) = (G_2,3)$ or $(F_4,2)$. 
\end{thm}

\begin{cor}\label{c:fixV}
The set $\Delta$ is empty for every case $X$ arising in Table \ref{tab:main}.
\end{cor}

\begin{proof}
If the inequality in Theorem \ref{t:fixV} is satisfied, then $\bigcap_i C_V(x_i)$ is nontrivial and thus $\la x_1, \ldots, x_t \ra$, and also $G(x)$, has a nontrivial fixed space on $V$ for all $x = (x_1, \ldots, x_t) \in X$. But $C_V(G) = 0$ and we conclude that $\Delta$ is empty.
\end{proof}

\section{Parabolic actions}\label{s:parab}

Define $G$ and $X = \mathcal{C}_1 \times \cdots \times \mathcal{C}_t$ as in the previous section and let $\mathcal{P} = \{P_1, \ldots, P_r\}$ be a set of representatives of the conjugacy classes of maximal parabolic subgroups of $G$, where $r$ is the rank of $G$. Fix a subgroup $H \in \mathcal{P}$ and let $\O = G/H$ be the corresponding coset variety. Following \cite{BGG1}, we can use a character-theoretic approach to compute the dimensions of the fixed point spaces $C_{\O}(y)$ for each unipotent element $y \in G$.

Let $\s$ be a Steinberg endomorphism of $G$ with finite fixed point subgroup $G_{\s} = G(q)$ for some $p$-power $q$ and let $y \in G$ be unipotent. By inspecting the relevant tables in \cite[Chapter 22]{LS_book}, we observe that $y^G \cap G(q)$ is non-empty and thus every unipotent class in $G$ has a representative in $G(q)$. We may assume $H$ is $\s$-stable, so $H_{\s}$ is a maximal parabolic subgroup of $G_{\s}$ and we can consider the corresponding permutation character $\chi = 1^{G_{\s}}_{H_{\s}}$. According to \cite[Lemma 2.4]{LLS2}, the character $\chi$ admits the following decomposition
\begin{equation}\label{e:dec}
\chi = \sum_{\phi \in \widehat{W}}n_{\phi}R_{\phi}
\end{equation}
where $\widehat{W}$ is the set of complex irreducible characters of the Weyl group  $W$ of $G$. Here the $R_{\phi}$ are almost characters of $G_{\s}$ and the coefficients are given by the inner products $n_{\phi} = \la 1^W_{W_H},\phi \ra$, where $W_H$ is the corresponding parabolic subgroup of $W$. In each case, the precise decomposition of $\chi$ as in \eqref{e:dec} is presented in \cite[Section 2]{LLS2}. 

The restriction of each almost character $R_{\phi}$ to unipotent elements yields the Green functions of $G_{\s}$, as defined by Deligne and Lusztig \cite{DL}. Building on earlier work due to Beynon-Spaltenstein, Lusztig, Malle and Shoji, the computation of the Green functions for exceptional groups of Lie type in all characteristics has very recently been completed by Geck and L\"{u}beck \cite{Geck, Lub}. This allows us to compute $\chi(z)$ for every $p$-element $z \in G_{\s}$ and in each case we obtain a polynomial in $q$. By considering the degrees of these polynomials and by appealing to Lang-Weil \cite{LW}, we can read off $\dim C_{\O}(y)$ for each unipotent element $y \in G$. More precisely, if we write $(y^G)_{\s} = \bigcup_j z_j^{G_{\s}}$ as a union of $G_{\s}$-classes, then $\dim C_{\O}(y)$ is the maximal degree of the polynomials $\chi(z_j)$. 

This approach allows us to compute $\dim C_{\O}(y)$ for every unipotent element $y \in G$. For example, suppose $G = F_4$, $p \geqs 3$, $H = P_1$ and $y$ is contained in the class labelled $B_2$. Then $\dim \O = 15$ and by inspecting \cite[p.414]{LLS2} we observe that 
\[
\chi = R_{\phi_{1,0}} + R_{\phi_{2,1}} + R_{\phi_{2,2}}+R_{\phi_{1,3}'}
\]
in terms of Carter's notation for irreducible characters of $W$ (see \cite{Car}). From \cite[Table 22.2.4]{LS_book}, we see that $(y^G)_{\s} = z_1^{G_{\s}} \cup z_2^{G_{\s}}$, where 
\[
|C_{G_{\s}}(z_1)| = 2q^{10}|{\rm SL}_2(q)|^2, \;\; |C_{G_{\s}}(z_2)| = 2q^{10}|{\rm SL}_2(q^2)|.
\]
By taking the appropriate Green functions, we compute
\[
\chi(z_1) = 2q^4+3q^3+2q^2+q+1,\;\; \chi(z_2) = q^3+q+1
\]
and we conclude that $\dim C_{\O}(y) = 4$.

By proceeding in this way, it is routine to verify the following result. In part (ii), $\tau$ is a graph automorphism of $G$. In each case, the dimension of $\O=G/H$ is recorded in \cite[Table 10]{BGS}.

\begin{thm}\label{t:par}
Let $V$ be the $kG$-module in Table \ref{tab:mod} and suppose $\Sigma_X(H) \geqs t-1$ for some maximal parabolic subgroup $H$ of $G$. 
Then one of the following holds:
\begin{itemize}\addtolength{\itemsep}{0.2\baselineskip}
\item[{\rm (i)}] The inequality in \eqref{e:di} is satisfied.
\item[{\rm (ii)}] $(G,p) = (G_2,3)$ or $(F_4,2)$, and \eqref{e:di} holds for
$\mathcal{C}_1^{\tau} \times \cdots \times \mathcal{C}_t^{\tau}$.
\item[{\rm (iii)}] $G = F_4$ and $(\mathcal{C}_1, \ldots, \mathcal{C}_t)$ is either $(A_1,\tilde{A}_1,(\tilde{A}_1)_2)$ or $(\tilde{A}_1,\tilde{A}_2)$, up to reordering.
\item[{\rm (iv)}] $G = E_6$, $E_7$ or $E_8$ and $(\mathcal{C}_1, \ldots, \mathcal{C}_t)$ is one of the cases recorded in Table \ref{tab:special}, up to reordering.
\end{itemize}
\end{thm}

Once we have proved Theorem \ref{t:main}, we can obtain the following corollary by applying Theorem \ref{t:par}.

\begin{cor}\label{c:parab}
Let $V$ be the $kG$-module in Table \ref{tab:mod}. Then $\Delta$ is empty if and only if one of the following holds, up to reordering and graph automorphisms:
\begin{itemize}\addtolength{\itemsep}{0.2\baselineskip}
\item[{\rm (i)}] The inequality in \eqref{e:di} is satisfied.
\item[{\rm (ii)}] $\Sigma_X(H) \geqs t-1$ for some maximal parabolic subgroup $H$ of $G$.
\item[{\rm (iii)}] $G=F_4$ and $(\mathcal{C}_1, \ldots, \mathcal{C}_t) = (A_1, \tilde{A}_1, \tilde{A}_1)$, $(A_1, (\tilde{A}_1)_2, (\tilde{A}_1)_2)$, $(\tilde{A}_1,A_2\tilde{A}_1)$ or $(\tilde{A}_1,B_2)$.
\end{itemize}
In particular, $\Delta$ is empty if and only if either \eqref{e:di} holds, or $\Sigma_X(H) \geqs t-1$ for some positive dimensional maximal subgroup $H$ of $G$.
\end{cor}

\begin{proof}
In view of Theorems \ref{t:main} and \ref{t:par}, the first statement follows by inspecting the cases appearing in Table \ref{tab:special} with $G=F_4$, excluding the two configurations in part (iii) of Theorem \ref{t:par}. And to complete the proof, it just remains to show that for each case in part (iii), there exists a positive dimensional maximal subgroup $H$ with $\Sigma_X(H) \geqs t-1$. If $(\mathcal{C}_1, \mathcal{C}_2,\mathcal{C}_3) = (A_1, \tilde{A}_1, \tilde{A}_1)$ then the proof of Proposition \ref{p:f4_dim3} yields $\Sigma_X(B_4) = 2$. Similarly, for $(A_1, (\tilde{A}_1)_2, (\tilde{A}_1)_2)$ we get $\Sigma_X(C_4) = 2$, while $\Sigma_X(B_4) = 1$ for $(\tilde{A}_1,A_2\tilde{A}_1)$ and $(\tilde{A}_1,B_2)$.
\end{proof}

\begin{rem}
In part (ii) of Corollary \ref{c:parab}, we may assume $(G,H)$ is one of the following:
\[
(F_4,P_1), \, (E_6, P_1), \, (E_7, P_1), \, (E_8,P_8),
\]
where the maximal parabolic subgroups are labelled in the usual manner. That is, if we are not in cases (i) or (iii), then $G \ne G_2$ and one can check that $\Sigma_X(H) \geqs t-1$ with respect to the specific maximal parabolic subgroup $H$ listed above.
\end{rem}

\section{Proof of Theorem \ref{t:main}}\label{s:main}

We are now ready to prove Theorem \ref{t:main}. Throughout this section, $G$ and $X$ are defined as in the statement of Theorem \ref{t:main} and recall that we adopt the notation for unipotent classes from \cite{LS_book}. We partition the proof into several subsections according to the group $G$.

\subsection{$G = G_2$}\label{ss:g2}

We begin by assuming $G = G_2$. Information on the unipotent conjugacy classes of $G$ is recorded in \cite[Table 22.1.5]{LS_book} and we adopt the notation therein for labelling the classes. By inspecting \cite[Table 1]{Law1}, it is easy to determine the required condition on $p$ to ensure that the elements in a given unipotent class have order $p$. In addition, note that if $p=3$ then a graph automorphism $\tau$ of $G$ interchanges the classes labelled $A_1$ and $\tilde{A}_1$ comprising long and short root elements, respectively, while the remaining classes are stable under $\tau$. 

Let $\mathcal{M}$ be a set of representatives of the conjugacy classes of closed positive dimensional maximal subgroups of $G$ and write $\mathcal{M} = \mathcal{P} \cup \mathcal{R}$, where the subgroups in $\mathcal{P}$ and $\mathcal{R}$ are parabolic and reductive, respectively. The subgroups in $\mathcal{R}$ are listed in Table \ref{tab:max}.

We will need the following result on fixed point spaces. Recall that the expression $\a(G,H,y)$ is defined in \eqref{e:alpha}. In Table \ref{tab:beta_g2}, $\delta_{r,p}$ is the Kronecker delta, so $\delta_{r,p}=1$ if $p=r$, otherwise $\delta_{r,p}=0$.

\begin{prop}\label{p:g2_dim3}
Let $H \in \mathcal{R}$ and let $y \in G$ be an element of order $p$. Then $\a(G,H,y) \leqs \b$, where $\b$ is recorded in Table \ref{tab:beta_g2}.
\end{prop}

{\small
\begin{table}
\[
\begin{array}{l|ccccc}
& A_1 & \tilde{A}_1 & (\tilde{A}_1)_3 & G_2(a_1) & G_2 \\ \hline
A_2.2 & 2/3 & \delta_{2,p}/2 & 0 & 1/3  & 0 \\
\tilde{A}_2.2  & 0 & 2/3 & 0 & 1/3 & 0 \\
A_1\tilde{A}_1 & 1/2 & (1+\delta_{2,p}+\delta_{3,p})/4 & 0 & 1/4 & 0 \\
A_1 & 0 & 0 & 0 & 0 & 1/11 \\
\end{array}
\]
\caption{The upper bound $\a(G,H,y) \leqs \b$ in Proposition \ref{p:g2_dim3}}
\label{tab:beta_g2}
\end{table}}

\begin{proof}
We consider each possibility for $H$ in turn, working with Proposition \ref{p:dim} to obtain the required upper bound on $\dim C_{\O}(y)$.

First assume $H = A_2.2$, so $\dim \O = 6$ and $H^0$ is generated by long root subgroups of $G$. Let $y \in G$ be an element of order $p$. Now the $G$-class of each unipotent class in $H^0$ is determined by Lawther in \cite[Section 4.2]{Law2} and as a consequence we compute $\dim(y^G \cap H^0) = 4,6$ if $y \in A_1, G_2(a_1)$, respectively, otherwise $\dim(y^G \cap H^0)=0$. If $p \geqs 3$, or if $p = 2$ and $y^G \cap (H \setminus H^0)$ is empty, then $y^G \cap H = y^G \cap H^0$  and we can calculate $\dim C_{\O}(y)$ via Proposition \ref{p:dim}. 

So to complete the analysis of this case, we may assume $p=2$ and $y \in H \setminus H^0$ is an involution. Here $y$ acts as a graph automorphism on $H^0$ (see \cite[p.3]{LS04}) and thus $C_{H^0}(y) = B_1$ by Proposition \ref{p:graph}. In order to determine the $G$-class of such an element $y$, let us consider the decomposition 
\[
V\downarrow H^0 = U \oplus U^* \oplus 0,
\]
where $V = W_G(\l_1)$ is the $7$-dimensional Weyl module for $G$ with highest weight $\l_1$ and $U$ and $0$ are the natural and trivial modules for $H^0$, respectively. Since $y$ interchanges the two $3$-dimensional summands, we deduce that $y$ has Jordan form $(J_2^3,J_1)$ on $V$ and by  inspecting \cite[Table 1]{Law1} we see that $y \in \tilde{A}_1$. Therefore, $\a(G,H,y) \leqs \delta_{2,p}/2$ when $y$ is in the class $\tilde{A}_1$.

Next assume $H = \tilde{A}_2.2$ and $p=3$, where $H^0$ is generated by short root subgroups. Here $H$ is the image of a maximal subgroup $A_2.2$ under a graph automorphism $\tau$ (where the connected component $A_2$ is generated by long root subgroups) and so the result follows immediately from our analysis of the previous case, recalling that $\tau$ interchanges the $G$-classes labelled $A_1$ and $\tilde{A}_1$. 
Finally, if $H = A_1\tilde{A}_1$ or $A_1$ then the $G$-class of each unipotent $H$-class is determined in \cite{Law2} and the desired result quickly follows.
\end{proof}

We are now in a position to prove Theorem \ref{t:main} for $G = G_2$. Recall that $\Delta$ is non-empty when $t \geqs 4$ by \cite[Theorem 7]{BGG1}, so we may assume $t \in \{2,3\}$.

\begin{thm}\label{t:g2}
The conclusion to Theorem \ref{t:main} holds when $G = G_2$. 
\end{thm}

\begin{proof}
First assume $t=3$. By Corollary \ref{c:fixV}, we know that $\Delta$ is empty when $\mathcal{C}_i = A_1$ for all $i$. By considering Proposition \ref{p:clos} and the closure relation on the set of unipotent classes in $G$ (see \cite[Section II.10]{Spal}), it suffices to show that $\Delta$ is non-empty when $X = \mathcal{C}_1 \times \mathcal{C}_2 \times \mathcal{C}_3$ with $\mathcal{C}_1 = \mathcal{C}_2 = A_1$ and $\mathcal{C}_3 = \tilde{A}_1$. By Proposition \ref{p:fix}, we just need to verify the bound $\Sigma_X(H)<2$ for all $H \in \mathcal{M}$, where $\Sigma_X(H)$ is defined as in \eqref{e:sig}. By Theorem \ref{t:par}, this bound holds when $H \in \mathcal{P}$ is a maximal parabolic subgroup. And for $H \in \mathcal{R}$, the desired result follows from the upper bounds on $\a(G,H,y)$ in Proposition \ref{p:g2_dim3}.

A very similar argument applies when $t=2$ and $p \geqs 3$. Here it suffices to show that $\Sigma_X(H)<1$ for all $H \in \mathcal{R}$ when $(\mathcal{C}_1,\mathcal{C}_2) = (A_1,G_2)$, $(\tilde{A}_1,(\tilde{A}_1)_3)$ or $(\tilde{A}_1,\tilde{A}_1)$, with $p \geqs 5$ in the latter case (if $p=3$ then $(\tilde{A}_1,\tilde{A}_1)$ is the image of $(A_1,A_1)$ under a graph automorphism). Once again, the result follows by applying the bounds presented in Proposition \ref{p:g2_dim3}.
\end{proof}

\subsection{$G = F_4$}\label{ss:f4}

Next assume $G = F_4$. Here we refer the reader to \cite[Table 22.1.4]{LS_book} for information on the unipotent classes in $G$, including the notation we use to label the classes. Note that if $p=2$ then a graph automorphism interchanges the $G$-classes labelled $A_1$ and $\tilde{A}_1$ comprising long and short root elements, respectively, and it fixes the classes $(\tilde{A}_1)_2$ and $A_1\tilde{A}_1$. As before, we write $\mathcal{M} = \mathcal{P} \cup \mathcal{R}$ for a set of representatives of the conjugacy classes of closed positive dimensional maximal subgroups of $G$ (see Table \ref{tab:max} for the subgroups in $\mathcal{R}$). We also write $\mathcal{L}$ for the subset of $\mathcal{R}$ defined in Table \ref{tab:long} (see Theorem \ref{t:long}). 

We begin by establishing the following result on fixed point spaces.

\begin{prop}\label{p:f4_dim3}
Let $H \in \mathcal{L}$ and let $y \in G$ be an element of order $p$ in one of the following conjugacy classes
\[
A_1, \, \tilde{A}_1, \, (\tilde{A}_1)_2, \, A_1\tilde{A}_1, \, A_2, \, \tilde{A}_2, \, A_2\tilde{A}_1, \, \tilde{A}_2A_1, \, C_3. 
\]
Then $\a(G,H,y) \leqs \b$, where $\b$ is recorded in Table \ref{tab:beta_f4}.
\end{prop}

{\small
\begin{table}
\[
\begin{array}{l|ccccccccc}
& A_1 & \tilde{A}_1 & (\tilde{A}_1)_2 & A_1\tilde{A}_1 & A_2 & \tilde{A}_2 & A_2\tilde{A}_1 & \tilde{A}_2A_1 & C_3 \\ \hline
B_4 & 3/4 & (5-\delta_{2,p})/8 & 5/8 & 1/2 & 1/2 & 0 & 3/8 & 0 & 0 \\
C_4 & 1/2 & 3/4 & 5/8 & 1/2 &  &  &  &  &  \\
A_1C_3 & 9/14 & 4/7 & & 3/7 & 3/7 & 3/7 & 0 & 2/7 & 1/7 \\
A_1G_2 & 5/7 & 0 & & 3/7 & 3/7 & 13/35 & 0 & 11/35 &  0 \\
A_2\tilde{A}_2.2 & 2/3 & (3+\delta_{2,p})/6 & 0 & 1/2 & 1/3 & 1/3 & 1/3 & 1/3 &  0 \\ 
D_4.S_3 & 3/4 & 5/8 & 7/12 & 1/2 & 1/2 & 1/3 & 0 & 1/3 &  0 \\
\tilde{D}_4.S_3 & 5/8 & 3/4 & 7/12 & 1/2 & & & & &  \\
\end{array}
\]
\caption{The upper bound $\a(G,H,y) \leqs \b$ in Proposition \ref{p:f4_dim3}}
\label{tab:beta_f4}
\end{table}}

\begin{proof}
If $H \in \{B_4,C_4,A_1C_3,A_1G_2\}$ then Lawther \cite{Law2} has determined the $G$-class of each $H$-class of unipotent elements and by appealing to Proposition \ref{p:dim} it is a straightforward exercise to compute $\a(G,H,y)$ in each case. For example, if $H = B_4$ and $p=2$, then the $G$-class of each $H$-class of involutions in $H$ is recorded in Table \ref{tab:b4} (in the first column, we use the notation from \cite[Table 4]{Law2} for the $H$-class of $y$, with the corresponding label from \cite{AS} given in the second column). As a consequence, we deduce that if $y \in A_1\tilde{A}_1$ then 
\[
\dim C_{\O}(y) = \dim \O - \dim y^G + \dim(y^G \cap H) = 16-28+20 = 8
\]
and thus $\a(G,H,y) = 1/2$.

{\small
\begin{table}
\[
\begin{array}{ccccc} \hline
\mbox{$H$-class of $y$} & & \dim y^H & \mbox{$G$-class of $y$} & \dim y^G \\ \hline
A_1 & a_2 & 12 & A_1 & 16 \\
B_1 & b_1 & 8 & \tilde{A}_1 & 16 \\
B_1^{(2)} & c_2 & 14 & (\tilde{A}_1)_2 & 22 \\
2A_1 & a_4 & 16 & (\tilde{A}_1)_2  & 22 \\
A_1+B_1 & b_3 & 18 & A_1\tilde{A}_1  & 28 \\
A_1+B_1^{(2)} & c_4 & 20 & A_1\tilde{A}_1 & 28 \\ \hline
\end{array}
\] 
\caption{The case $G = F_4$, $H = B_4$, $p=2$}
\label{tab:b4}
\end{table}}

Next assume $H = A_2\tilde{A}_2.2$. The $G$-class of each $H^0$-class of unipotent elements is determined in \cite[Section 4.7]{Law2} and this allows us to  compute $\dim(y^G \cap H^0)$. This gives $\dim(y^G \cap H)$, and hence $\dim C_{\O}(y)$ via Proposition \ref{p:dim}, unless $p=2$ and $y^G \cap (H \setminus H^0)$ is non-empty. So we may assume $p=2$ and $y \in H \setminus H^0$ is an involution. Here $y$ induces a graph automorphism on both $A_2$ factors of $H^0$ and thus $C_{H^0}(y) = B_1^2$ (see Proposition \ref{p:graph}). In order to identify the $G$-class of $y$, let $V = W_G(\l_4)$ be the $26$-dimensional Weyl module with highest weight $\l_4$ and note that 
\[
V\downarrow H^0 = (U \otimes U) \oplus (U^* \otimes U^*) \oplus (0 \otimes \mathcal{L}(A_2))
\]
(see \cite[Table 2]{Thomas}), where $U$, $\mathcal{L}(A_2)$ and $0$ are the natural, adjoint and trivial modules for $A_2$, respectively. Now $y$ interchanges the two $9$-dimensional summands and we calculate that it has Jordan form $(J_2^3,J_1^2)$ on $\mathcal{L}(A_2)$ (to do this, one just needs to consider the action of the transpose map on the space of trace-zero $3 \times 3$ matrices over $k$). Therefore, $y$ has Jordan form $(J_2^{12},J_1^2)$ on $V$ and by inspecting \cite[Table 3]{Law1} we deduce that $y$ is in the $G$-class labelled $A_1\tilde{A}_1$. Since $\dim(y^G \cap H^0) = 8$, we deduce that $\dim(y^G \cap H) = 10$ and thus $\a(G,H,y) = 1/2$.

Next suppose $H = D_4.S_3$. Here $H^0 < B_4 < G$ and so we can use \cite[Section 4.4]{Law2} to compute $\dim(y^G \cap H^0)$. For example, suppose $p=2$ and observe that there are three $H$-classes of involutions in $H^0$, represented by the elements $a_2$, $c_2$ and $c_4$ in the notation of \cite{AS} (note that involutions of type $c_2$, $a_4$ and $a_4'$ are conjugate under a triality graph automorphism of $H^0$). The corresponding class in $B_4$ has the same label and the $G$-class can be read off from Table \ref{tab:b4}. 

So to complete the analysis of this case, we may assume $p \in \{2,3\}$ and $y \in H \setminus H^0$ has order $p$. First assume $p=2$. Here we can proceed as above, noting that $D_4.2 < B_4$ and the relevant involutions are of type $b_1$ and $b_3$, which we view as graph automorphisms of $H^0$. By consulting Table \ref{tab:b4}, we see that the $b_1$-involutions are contained in the $G$-class $\tilde{A}_1$, while those of type $b_3$ are in the class labelled $A_1\tilde{A}_1$. So for $y \in \tilde{A}_1$ we deduce that $\dim(y^G \cap H) \leqs 7$ and thus $\a(G,H,y) \leqs 5/8$. Similarly, if $y \in A_1\tilde{A}_1$ then $\dim(y^G \cap H) = 16$ and $\a(G,H,y) = 1/2$. 

Now assume $p=3$ and $y \in H \setminus H^0$ has order $3$. Here $y$ acts as a triality graph automorphism on $H^0$ and there are two $H^0$-classes to consider, represented by $y_1$ and $y_2$, where $C_{H^0}(y_1) = G_2$ and $C_{H^0}(y_2) = C_{G_2}(u)$ with $u \in G_2$ a long root element (see  Proposition \ref{p:graph}). As above, let $V = W_G(\l_4)$ and note that  
\[
V\downarrow H^0 = U_1 \oplus U_2 \oplus U_3 \oplus 0^2,
\]
where the $U_i$ denote the three $8$-dimensional irreducible modules for $H^0$ (namely, the natural module and the two spin modules) and $0$ is the trivial module (see \cite[Table 2]{Thomas}). Now $y$ cyclically permutes $U_1,U_2$ and $U_3$, whence  the Jordan form of $y$ on $V$ has $8$ Jordan blocks of size $3$. By inspecting \cite[Table 3]{Law1}, this places $y$ in the $G$-class labelled $\tilde{A}_2$ or $\tilde{A}_2A_1$. In fact, by arguing as in the proof of \cite[Proposition 5.14]{BGS} we can show that $y \in \tilde{A}_2A_1$ when $C_{H^0}(y) = C_{G_2}(u)$. Therefore, we conclude that $\dim(y^G \cap H) \leqs 14$ and $\a(G,H,y) \leqs 1/3$ if $y \in \tilde{A}_2$. Similarly, we get $\a(G,H,y) \leqs 1/3$ if $y \in \tilde{A}_2A_1$.

Finally, let us observe that the result for $H =  \tilde{D}_4.S_3$ with $p=2$ follows immediately from our analysis of the previous case, noting that $H$ is the image of a maximal subgroup $D_4.S_3$ of $G$ under a graph automorphism, where the connected component of the latter group is generated by long root subgroups. 
\end{proof}

If $t \geqs 5$ then $\Delta$ is non-empty by \cite[Theorem 7]{BGG1}, so we may assume $t \in \{2,3,4\}$. First we handle the case $t=4$.

\begin{thm}\label{t:f4_4}
The conclusion to Theorem \ref{t:main} holds when $G = F_4$ and $t=4$.
\end{thm}

\begin{proof}
If $\mathcal{C}_i = A_1$ for all $i$ then $\Delta$ is empty by Corollary \ref{c:fixV}, and the same conclusion holds if $p=2$ and $\mathcal{C}_i = \tilde{A}_1$ for all $i$. Therefore, it remains to show that $\Delta$ is non-empty in all other cases. In view of Proposition \ref{p:fix}, it suffices to show that $\Sigma_X(H)<3$ for all $H \in \mathcal{M}$. 

Let $y \in G$ be a unipotent element of order $p$ and let $H \in \mathcal{M}$. By \cite[Theorem 3.1]{BGG1} we have $\a(G,H,y) \leqs 3/4$ if $y \in A_1$, or if $p=2$ and $y \in \tilde{A}_1$, otherwise $\a(G,H,y)<2/3$. Therefore,
\[
\Sigma_X(H) < 3\cdot \frac{3}{4}+\frac{2}{3} < 3
\]
and the result follows.
\end{proof}

Now assume $t=3$. We begin by showing that $\Delta$ is empty for the three special cases recorded in Table \ref{tab:special}.

\begin{lem}\label{l:f4_30}
 The set $\Delta$ is empty if $t=3$, $p=2$ and $(\mathcal{C}_1,\mathcal{C}_2,\mathcal{C}_3) = (\tilde{A}_1,(\tilde{A}_1)_2,(\tilde{A}_1)_2)$ or $(A_1,\tilde{A}_1,(\tilde{A}_1)_2)$.
\end{lem}

\begin{proof}
In view of Proposition \ref{p:clos}, noting that $A_1$ is contained in the closure of $(\tilde{A}_1)_2$ (see \cite{Spal} and Remark \ref{r:spal}), it suffices to show that $\Delta$ is empty for $(\mathcal{C}_1,\mathcal{C}_2,\mathcal{C}_3) = (\tilde{A}_1,(\tilde{A}_1)_2,(\tilde{A}_1)_2)$. Write $\mathcal{C}_i = y_i^G$ and observe that we may embed each $y_i$ in a maximal closed subgroup $L = C_4$, where $y_1$ is an $a_2$-type involution and $y_2, y_3$ are of type $a_4$ (see \cite[Section 4.5]{Law2}). If $W$ denotes the natural module for $L$, then
\[
\sum_{i=1}^3 \dim C_W(y_i) = 6+4+4 = 14 < 2\dim W
\]
and by applying \cite[Theorem 7]{BGG2} it follows that we may assume $G(y) = L$. If we set $\mathcal{D}_i = y_i^L$ then $\dim \mathcal{D}_1 = 12$ and $\dim \mathcal{D}_i = 16$ for $i=2,3$, whence
\[
\dim X - \dim (\mathcal{D}_1 \times \mathcal{D}_2 \times \mathcal{D}_3)  = 16 = \dim G - \dim L.
\]
Now $N_G(L) = L$ and one can check that all of the conditions in Proposition \ref{p:fibre} are satisfied (with $M=L$). We conclude that $\Delta$ is empty.
\end{proof}

\begin{lem}\label{l:f4_300}
The set $\Delta$ is empty if $t=3$ and $(\mathcal{C}_1,\mathcal{C}_2,\mathcal{C}_3) = (A_1, \tilde{A}_1, \tilde{A}_1)$.
\end{lem}

\begin{proof}
This is clear when $p=2$ since we know by Corollary \ref{c:fixV} that $\Delta$ is empty for the triple $(\tilde{A}_1, A_1,A_1)$, which is the image of $(\mathcal{C}_1,\mathcal{C}_2,\mathcal{C}_3)$ under a graph automorphism. Now assume $p \geqs 3$ and embed each $y_i$ in a maximal closed subgroup $L = B_4$, where $y_1$ has Jordan form $(J_2^2,J_1^5)$ on the natural module for $L$, while $y_2$ and $y_3$ both have Jordan form $(J_2^4,J_1)$. Set $\mathcal{D}_i = y_i^L$ and note that 
$N_G(L) = L$ and 
\[
\dim \mathcal{C}_1 = 16, \; \dim \mathcal{C}_2 = \dim \mathcal{C}_3 = 22, \; \dim \mathcal{D}_1 = 12, \; \dim \mathcal{D}_2 = \dim \mathcal{D}_3 = 16.
\]
In addition, by \cite[Theorem 7]{BGG2}, we may assume that the $y_i$ topologically generate $L$. Setting $Y = \mathcal{D}_1 \times \mathcal{D}_2 \times \mathcal{D}_3$ we compute $\dim X - \dim Y = 16 = \dim G - \dim L$ and thus $\Delta$ is empty by Proposition \ref{p:fibre}.
\end{proof}

\begin{thm}\label{t:f4_3}
The conclusion to Theorem \ref{t:main} holds when $G = F_4$ and $t=3$.
\end{thm}

\begin{proof}
Set $\mathcal{C}_i = y_i^G$. In view of Corollary \ref{c:fixV} and Lemmas \ref{l:f4_30} and \ref{l:f4_300}, we just need to show that $\Delta$ is non-empty when $(\mathcal{C}_1,\mathcal{C}_2,\mathcal{C}_3)$ is not one of the triples in Tables \ref{tab:main} or \ref{tab:special}, up to reordering and graph automorphisms when $p=2$. To do this, we will verify the bound $\Sigma_X(H)<2$ for all $H \in \mathcal{M}$ and apply Proposition \ref{p:fix}. By Theorem \ref{t:par}, this bound holds when $H \in \mathcal{P}$, so we may assume $H \in \mathcal{R}$.

First assume $\mathcal{C}_1 = \mathcal{C}_2 = A_1$. Here $p \geqs 3$ and by considering the closure relation on unipotent classes, it suffices to show that $\Sigma_X(H)<2$ when $\mathcal{C}_3 = \tilde{A}_2$ or $A_2\tilde{A}_1$. If $H \not\in \mathcal{L}$ then $H \cap \mathcal{C}_i$ is empty for $i=1,2$, so $\Sigma_X(H) = \a(G,H,y_3)$ and the result follows. Therefore, we may assume $H \in \mathcal{L}$ and one can check that the bounds in Proposition \ref{p:f4_dim3} are sufficient. Next suppose  $\mathcal{C}_1 = A_1$ and $\mathcal{C}_2 \ne A_1$. Here it is sufficient to show that $\Sigma_X(H)<2$ for the triple $(A_1,\tilde{A}_1,A_1\tilde{A}_1)$ and once again we find that the bounds in Proposition \ref{p:f4_dim3} are good enough. Finally, if $\mathcal{C}_1 \ne A_1$ (and also $\mathcal{C}_1 \ne \tilde{A}_1$ if $p=2$), then \cite[Theorem 3.1]{BGG1} implies that $\a(G,H,y_i)<2/3$ for all $H \in \mathcal{M}$ and thus $\Sigma_X(H)<2$.
\end{proof}

Finally, let us assume $t=2$. First we handle the special cases in Table \ref{tab:special}. 

\begin{lem}\label{l:f4_31}
The set $\Delta$ is empty if $t=2$ and $(\mathcal{C}_1,\mathcal{C}_2) = (\tilde{A}_1,\tilde{A}_2)$, $(\tilde{A}_1,A_2\tilde{A}_1)$ or $(\tilde{A}_1,B_2)$.
\end{lem}

\begin{proof}
Write $\mathcal{C}_i = y_i^G$ and first assume $(\mathcal{C}_1,\mathcal{C}_2) = (\tilde{A}_1,\tilde{A}_2)$. We may embed $y_1,y_2 \in L = C_3$, where $M = N_G(L) = A_1C_3$ is a maximal closed subgroup of $G$. Here $y_1$ and $y_2$ have respective Jordan forms $(J_2^2,J_1^2)$ and $(J_3^2)$ on the natural module for $L$, so \cite[Theorem 7]{BGG2} implies that $G(x) = L$ for some $x \in X$. Let $\mathcal{D}_i = y_i^M$ and note that $\dim \mathcal{D}_1 = 10$ and $\dim \mathcal{D}_2 = 14$. Then  
\[
\dim X - \dim(\mathcal{D}_1 \times \mathcal{D}_2) = 28 = \dim G - \dim M
\]
and we now conclude by applying Proposition \ref{p:fibre}.

To complete the proof, it suffices to show that $\Delta$ is empty when $(\mathcal{C}_1,\mathcal{C}_2) = (\tilde{A}_1,B_2)$ since the class $A_2\tilde{A}_1$ is contained in the closure of the class labelled $B_2$. Here we embed $y_1,y_2 \in L = B_4$, with respective Jordan forms $(J_2^4,J_1)$ and $(J_4^2,J_1)$ on the natural module for $L$ and the reader can check that the same argument applies, via \cite[Theorem 7]{BGG2} and Proposition \ref{p:fibre} (note that $L$ is maximal, so $M = N_G(L) = L$).
\end{proof}

\begin{rem}\label{r:f4_lie}
Let us consider the cases in Lemma \ref{l:f4_31} under the assumption that the field $k$ has arbitrary characteristic $p \geqs 0$. In view of \cite[Corollary 3.20]{BGG2}, we see that $\Delta$ is non-empty only if 
\begin{equation}\label{e:lieX}
\dim X \geqs \dim G + {\rm rk}(G) - \dim Z(\mathcal{L}(G)),
\end{equation}
where ${\rm rk}(G)$ and $Z(\mathcal{L}(G))$ denote the rank of $G$ and the centre of the Lie algebra of $G$, respectively. Since $Z(\mathcal{L}(G)) = 0$ for $G = F_4$, it follows that $\Delta$ is non-empty only if $\dim X \geqs 56$. In particular, if 
$(\mathcal{C}_1,\mathcal{C}_2) = (\tilde{A}_1,\tilde{A}_2)$ then 
$\dim X = 52 - 6\delta_{2,p}$ and $\Delta$ is empty. On the other hand, if $p \ne 2$ and $(\mathcal{C}_1,\mathcal{C}_2) = (\tilde{A}_1,A_2\tilde{A}_1)$ or $(\tilde{A}_1,B_2)$, then $\dim X = 56$ and so this approach via \cite[Corollary 3.20]{BGG2} is inconclusive in these two cases.
\end{rem}

\begin{thm}\label{t:f4_2}
The conclusion to Theorem \ref{t:main} holds when $G = F_4$, $t=2$ and $p \geqs 3$.
\end{thm}

\begin{proof}
By combining Corollary \ref{c:fixV} and Lemma \ref{l:f4_31}, it remains to show that $\Delta$ is non-empty for every pair $(\mathcal{C}_1,\mathcal{C}_2)$ not appearing in Tables \ref{tab:main} and \ref{tab:special} (as usual, up to reordering and graph automorphisms). In view of Proposition \ref{p:fix} and Theorem \ref{t:par}, it suffices to verify the bound $\Sigma_X(H)<1$ for all $H \in \mathcal{R}$.

Note that if $H \in \mathcal{R} \setminus \mathcal{L}$, then $H = G_2$ (with $p=7$) or $H = A_1$ and $p \geqs 13$ (see Tables \ref{tab:max} and \ref{tab:long}). By \cite[Section 5.2]{Law2}, $H = G_2$ meets the nontrivial unipotent classes labelled $A_1\tilde{A}_1$, $\tilde{A}_2A_1$, $F_4(a_3)$ and $F_4(a_2)$. Similarly, the nontrivial unipotent elements in $H = A_1$ are in the $G$-class labelled $F_4$ (see \cite[Table 27]{Law2}).

Suppose $\mathcal{C}_1 = A_1$. As before, if $H \not\in \mathcal{L}$ then $\a(G,H,y_1) = 0$ for $y_1 \in \mathcal{C}_1$ and the desired bound follows. So we only need to consider the subgroups in $\mathcal{L}$. By carefully reviewing the closure relation on unipotent classes, we may assume $\mathcal{C}_2$ is the class labelled $C_3$ and the desired bound $\Sigma_X(H)<1$ follows from Proposition \ref{p:f4_dim3}. Similarly, if $\mathcal{C}_1 = \tilde{A}_1$ then we may assume $\mathcal{C}_2 = \tilde{A}_2A_1$ and $H \in \mathcal{L}$, finding once again that the bounds in Proposition \ref{p:f4_dim3} are good enough. The same argument applies if $\mathcal{C}_1 \in \{ A_1\tilde{A}_1, A_2\}$, noting that in each case we may assume $\mathcal{C}_2 \in \{\tilde{A}_2, A_2\tilde{A}_1\}$. 

Finally, let us assume $\mathcal{C}_i \not\in \{A_1, \tilde{A}_1, A_1\tilde{A}_1, A_2\}$ for $i=1,2$. By considering closures, it suffices to show that $\Delta$ is non-empty when each $\mathcal{C}_i$ is contained in $\{\tilde{A}_2, A_2\tilde{A}_1\}$ and it is easy to check that the bounds in Proposition \ref{p:f4_dim3} are sufficient.
\end{proof}

This completes the proof of Theorem \ref{t:main} for $G = F_4$.

\subsection{$G = E_6$}\label{ss:e6}

In this section we assume $G = E_6$. We adopt the labelling of unipotent classes given in \cite[Table 22.1.3]{LS_book} and we define the subgroup collections $\mathcal{M}$, $\mathcal{P}$, $\mathcal{R}$ and $\mathcal{L}$ as in Section \ref{ss:sub}. We begin with the following result on fixed point spaces. 

\begin{prop}\label{p:e6_dim3}
Let $H \in \mathcal{L}$ and let $y \in G$ be an element of order $p$ in one of the following conjugacy classes
\[
A_1, \, A_1^2, \, A_1^3, \, A_2, \, A_2A_1, \, A_2^2A_1, \, A_4A_1. 
\]
Then $\a(G,H,y) \leqs \b$, where $\b$ is recorded in Table \ref{tab:beta_e6}.
\end{prop}

{\small
\begin{table}
\[
\begin{array}{l|ccccccc}
& A_1 & A_1^2 & A_1^3 & A_2 & A_2A_1 & A_2^2A_1 & A_4A_1  \\ \hline
A_1A_5 & 7/10 & 3/5 & 1/2 & 9/20 & 2/5 & 3/10 & 1/5 \\
F_4 & 10/13 & 8/13 & 7/13 & 7/13 & 0 & 4/13 & 0 \\
C_4 & 2/3 & 4/7 & 10/21 & 10/21 & 0 & 2/7 & 0 \\
A_2G_2 & 5/7 & 1/2 & 1/2 & 3/7 & 11/28 & 9/28 & 0 \\
A_2^3.S_3 & 2/3 & 5/9 & 1/2 & 1/3 & 1/3 & 8/27 & 0 \\
D_4T_2.S_3 & 3/4 & 7/12 & 25/48 & 1/2 & 0 & 1/3 & 0 \\
T.W & 17/24 & 23/36 & 19/36 & 1/2 & 4/9 & 1/3 & 2/9 \\ 
 \end{array}
\]
\caption{The upper bound $\a(G,H,y) \leqs \b$ in Proposition \ref{p:e6_dim3}}
\label{tab:beta_e6}
\end{table}}

\begin{proof}
If $y \in A_1$ is a long root element then the upper bound on $\a(G,H,y)$ in Table \ref{tab:beta_e6} follows immediately from \cite[Theorem 3.1]{BGG1}. Now assume $y \not\in A_1$. If $H \in \{A_1A_5, F_4, C_4, A_2G_2\}$ then Lawther \cite{Law2} has determined the $G$-class of each $H$-class of unipotent elements and it is an easy exercise to compute $\a(G,H,y)$ in each case.

Next assume $H = A_2^3.S_3$. The $G$-class of each unipotent class in $H^0$ is determined in \cite[Section 4.9]{Law2} and this allows us to compute $\dim(y^G \cap H^0)$. Therefore, to complete the analysis of this case we may assume $p \in \{2,3\}$ and $y \in H \setminus H^0$ has order $p$. If $p=2$ then the proof of \cite[Lemma 4.12]{BTh0} shows that $C_{H^0}(y) = A_2B_1$ and $y$ is contained in the $G$-class labelled $A_1^3$. So for $y \in A_1^3$ we conclude that $\dim(y^G \cap H^0) = 12$ and 
$\dim(y^G \cap (H \setminus H^0)) \leqs 13$, whence $\a(G,H,y) \leqs 1/2$. Similarly, if $p=3$ and $y \in H \setminus H^0$ has order $3$ then $y$ is contained in the $G$-class $A_2^2$, which is not one of the classes we need to consider.

Now suppose $H = D_4T_2.S_3$.  Set $J = (H^0)' = D_4$ and note that $J.S_3 < F_4 < G$. We studied the embedding $J.S_3 < F_4$ in the proof of Proposition \ref{p:f4_dim3} and we note that the $G$-class of each unipotent class in $F_4$ is recorded in  \cite[Table 22.1.4]{LS_book}. In this way, it is straightforward to determine the $G$-class of each unipotent class in $H^0$ and then compute $\dim(y^G \cap H^0)$.
So we may assume $p \in \{2,3\}$ and $y \in H \setminus H^0$ has order $p$.

First assume $p=2$. In the notation of \cite{AS}, we may view $y$ as a $b_1$ or $b_3$ type involution in $D_4.2 = {\rm O}_8(k)$ and we recall that the $b_1$-involutions are contained in the $F_4$-class $\tilde{A}_1$ (which in turn places $y$ in the $G$-class labelled $A_1^2$), while those of type $b_3$ are in the $A_1\tilde{A}_1$ class of $F_4$, which is contained in the $A_1^3$ class of $G$. As a consequence, we deduce that $\dim C_{\O}(y) = 36, 28$ if $y \in A_1, A_1^2$, respectively. And if $y \in A_1^3$ then $\dim(y^G \cap H^0) = 16$ and $\dim(y^G \cap (H \setminus H^0)) \leqs 17$, which implies that $\dim C_{\O}(y) \leqs 25$ and thus $\a(G,H,y) \leqs 25/48$ as recorded in Table \ref{tab:beta_e6}.

Now suppose $p=3$, in which case $y$ acts as a triality graph automorphism on $J = D_4$. By inspecting the proof of Proposition \ref{p:f4_dim3} we see that $y$ is in the $F_4$-class $\tilde{A}_2$ if $C_J(y) = G_2$ (in which case, $y$ is in the $G$-class labelled $A_2^2$), otherwise $y$ is in the $F_4$-class $\tilde{A}_2A_1$, which places $y$ in the $G$-class $A_2^2A_1$. So for $y \in A_2^2A_1$ we get $\dim(y^G \cap H^0) = 0$ and $\dim(y^G \cap (H \setminus H^0)) \leqs 22$, which yields $\a(G,H,y) \leqs 1/3$. 

Finally, if $H = T.W$ is the normaliser of a maximal torus, then the results in Table \ref{tab:beta_e6} (for $y \not\in A_1$) follow from the trivial bound $\dim(y^G \cap H) \leqs \dim H = 6$.
\end{proof}

\begin{thm}\label{t:e6_4}
The conclusion to Theorem \ref{t:main} holds when $G = E_6$ and $t=4$.
\end{thm}

\begin{proof}
By Corollary \ref{c:fixV} we know that $\Delta$ is empty when $\mathcal{C}_i = A_1$ for all $i$. In the remaining cases, it suffices to show that $\Sigma_X(H)<3$ for all $H \in \mathcal{M}$. By \cite[Theorem 3.1]{BGG1} we have $\a(G,H,y) \leqs 10/13$ if $y \in A_1$, otherwise $\a(G,H,y)<2/3$, whence $\Sigma_X(H) < 3\cdot 10/13+2/3$ and the result follows.
\end{proof}

\begin{thm}\label{t:e6_3}
The conclusion to Theorem \ref{t:main} holds when $G = E_6$ and $t=3$.
\end{thm}

\begin{proof}
In view of Corollary \ref{c:fixV}, we see that $\Delta$ is empty for each triple in Table \ref{tab:main}. Therefore, in the remaining cases it suffices to show that $\Sigma_X(H)<2$ for all $H \in \mathcal{M}$. By Theorem \ref{t:par}, we only need to check this for $H \in \mathcal{R}$. Set $\mathcal{C}_i = y_i^G$. 

If $\mathcal{C}_i \ne A_1$ for all $i$, then \cite[Theorem 3.1]{BGG1} implies that $\a(G,H,y_i)<2/3$ and thus $\Sigma_X(H) <2$.  Therefore, for the remainder of the proof we may assume $\mathcal{C}_1 = A_1$ and $H \in \mathcal{L}$. 

First assume $\mathcal{C}_2 = A_1$. By inspecting \cite{Spal}, we see that the class labelled $A_2A_1$ is contained in the closure of $\mathcal{C}_3$, so in view of Proposition \ref{p:clos} we may assume $\mathcal{C}_3 = A_2A_1$. One can now check that the bounds in Proposition \ref{p:e6_dim3} imply that $\Sigma_X(H)<2$. Finally, if $\mathcal{C}_2 \ne A_1$ then by considering closures we may assume $\mathcal{C}_2 = A_1^2$ and $\mathcal{C}_3 = A_1^3$. As before, the result now follows by applying the bounds presented in Proposition \ref{p:e6_dim3}.
\end{proof}

Finally, let us assume $t=2$. First we handle the special case from Table \ref{tab:special}.

\begin{lem}\label{l:e6_3}
The set $\Delta$ is empty if $t=2$ and $(\mathcal{C}_1,\mathcal{C}_2) = (A_1^2,A_2^2)$. 
\end{lem}

\begin{proof}
Write $\mathcal{C}_i = y_i^G$ and observe that we may embed $y_1$ and $y_2$ in a subgroup $L = A_5$ with $M = N_G(L) = A_1A_5$ so that the respective Jordan forms on the natural module for $L$ are $(J_2^2,J_1^2)$ and $(J_3^2)$. Then by applying the main theorem of \cite{Ger}, we may assume $G(y) = L$ for $y = (y_1,y_2) \in X$. Setting $\mathcal{D}_i = y_i^M$ and $Y = \mathcal{D}_1 \times \mathcal{D}_2$ we compute $\dim Y = 40$ and thus 
\[
\dim X - \dim Y = 40 = \dim G - \dim M.
\]
We now conclude via Proposition \ref{p:fibre}.
\end{proof}

\begin{thm}\label{t:e6_2}
The conclusion to Theorem \ref{t:main} holds when $G = E_6$, $t=2$ and $p \geqs 3$.
\end{thm}

\begin{proof}
Write $\mathcal{C}_i = y_i^G$. By combining Corollary \ref{c:fixV} and Lemma \ref{l:e6_3}, we have already shown that $\Delta$ is empty for all of the pairs $(\mathcal{C}_1,\mathcal{C}_2)$ recorded in Tables \ref{tab:main} and \ref{tab:special}. So in view of Theorem \ref{t:par}, it suffices to verify the bound $\Sigma_X(H)<1$ in each of the remaining cases, for every subgroup $H \in \mathcal{R}$. 

First assume $\mathcal{C}_1 = A_1$. If $H \not\in \mathcal{L}$ then $\Sigma_X(H) = \a(G,H,y_2)$ and the result follows, so we may as well assume $H \in \mathcal{L}$. In addition, by  considering the closure relation on unipotent classes, we only need to check that $\Sigma_X(H)<1$ when $\mathcal{C}_2$ is the class labelled $A_4A_1$. The result now follows by applying the relevant upper bounds in Proposition \ref{p:e6_dim3}. 

Next suppose $\mathcal{C}_1 = A_1^2$. Here we may assume $\mathcal{C}_2 = A_2^2A_1$ and by inspecting \cite{Law2} we see that $H$ meets $\mathcal{C}_1$ only if $H \in \mathcal{L}$. So we are free to assume that $H \in \mathcal{L}$ and we can now verify the inequality $\Sigma_X(H)<1$ from the bounds in Proposition \ref{p:e6_dim3}. Similarly, if $\mathcal{C}_1 = A_1^3$ or $A_2$ then we may assume $\mathcal{C}_2 = A_2A_1$ and $H \in \mathcal{L}$, noting once again that the desired result follows via Proposition \ref{p:e6_dim3} (note that if $H \not\in \mathcal{L}$, then either $\mathcal{C}_1$ or $\mathcal{C}_2$ does not meet $H$ and the bound $\Sigma_X(H)<1$ clearly holds).

Finally, let us assume $\mathcal{C}_i \not\in \{A_1,A_1^2,A_1^3,A_2\}$ for $i=1,2$. Here we may assume $\mathcal{C}_1 = \mathcal{C}_2 = A_2A_1$ and $H \in \mathcal{L}$, in which case the result follows in the usual fashion via Proposition \ref{p:e6_dim3}.
\end{proof}

This completes the proof of Theorem \ref{t:main} for $G = E_6$.

\subsection{$G = E_7$}\label{ss:e7}

In this section we prove Theorem \ref{t:main} for $G = E_7$. The unipotent classes in $G$ and their respective dimensions are recorded in \cite[Table 22.1.2]{LS_book} and we adopt the labelling of classes therein. We remark that there are several differences between our choice of labels and those used by other authors. For example, the classes $(A_1^3)^{(1)}$, $(A_3A_1)^{(1)}$ and $(A_5)^{(1)}$ are respectively labelled $(3A_1)''$, $(A_3+A_1)''$ and $(A_5)''$ in \cite{Law2,Law1,Spal}. We define the subgroup collections $\mathcal{M}$, $\mathcal{P}$, $\mathcal{R}$ and $\mathcal{L}$ as in Section \ref{ss:sub}. 

We begin with the following result on fixed point spaces.

\begin{prop}\label{p:e7_dim3}
Let $H \in \mathcal{L}$ and let $y \in G$ be an element of order $p$ in one of the following conjugacy classes
\[
A_1, \, A_1^2, \, (A_1^3)^{(1)}, \, (A_1^3)^{(2)}, \, A_1^4, \, A_2, \, A_2A_1^2, \, A_2^2A_1, \, A_4A_2.
\]
Then $\a(G,H,y) \leqs \b$, where $\b$ is recorded in Table \ref{tab:beta_e7}.
\end{prop}

{\small
\begin{table}
\[
\begin{array}{l|ccccccccc}
& A_1 & A_1^2 & (A_1^3)^{(1)} & (A_1^3)^{(2)} & A_2 & A_1^4 & A_2A_1^2 & A_2^2A_1 & A_4A_2 \\ \hline
A_1D_6 & 3/4 & 5/8 & 5/8 & 1/2 & 1/2 & (15+\delta_{2,p})/32 & 3/8 & 5/16 & 0 \\
A_1F_4 & 10/13 & 8/13 & 7/13 & 7/13 & 7/13 & 19/39 & 5/13 & 1/3 & 0 \\
G_2C_3 & 5/7 & 29/49 & 4/7 & 25/49 & 3/7 & 24/49 & 18/49 & 16/49 & 0 \\
A_7.2 & 5/7 & 3/5 & 43/70 & 19/35 & 18/35 & \delta_{2,p}/2 & 13/35 & 11/35 &1/5 \\
A_2A_5.2 & 11/15 & 3/5 & 3/5 & 23/45 & 7/15 & (14+\delta_{2,p})/30 & 17/45 & 1/3 & 1/5 \\
T_1E_6.2 & 7/9 & 17/27 & 1/2 & 5/9 & 5/9 & \delta_{2,p}/2 & 11/27 & 1/3 & 0 \\
A_1^3D_4.S_3 & 3/4 & 7/12 & 31/48 & 25/48 & 1/2 & (11+\delta_{2,p})/24 & 3/8 & 1/3 & 0 \\
A_1^7.{\rm GL}_3(2) & 5/7 & 37/56 & 9/14 & 31/56 & 15/28 & (27+\delta_{2,p})/56 & 11/28 & 9/28 & 5/28 \\
T.W & 31/42 & 9/14 & 79/126 & 23/42 & 67/126 & \delta_{2,p}/2 & 17/42 & 43/126 & 3/14 \\ 
\end{array}
\]
\caption{The upper bound $\a(G,H,y) \leqs \b$ in Proposition \ref{p:e7_dim3}}
\label{tab:beta_e7}
\end{table}}

\begin{proof}
If $y \in A_1$ is a long root element then the upper bound on $\a(G,H,y)$ in Table \ref{tab:beta_e7} follows from \cite[Theorem 3.1]{BGG1}, so for the remainder we may assume $y \not\in A_1$. Now if $H$ is one of the subgroups $A_1D_6, A_1F_4$ or $G_2C_3$, then the $G$-class of each $H$-class of unipotent elements is determined in \cite{Law2} and it is straightforward to verify the result via Proposition \ref{p:dim}.

The subgroups $A_7.2$ and $A_2A_5.2$ can be handled in a similar fashion, with some additional work when $p=2$. Indeed, by inspecting \cite[Sections 4.11, 4.12]{Law2} we can compute $\dim(y^G \cap H^0)$, so the analysis is reduced to the situation where $p=2$ and $y \in H \setminus H^0$ is an involution. 

First assume $H = A_7.2$. Here $y$ induces a graph automorphism on $H^0 = A_7$ and there are two $H^0$-classes of such elements, represented by $y_1$ and $y_2$, where 
$C_{H^0}(y_1) = C_4$ and $C_{H^0}(y_2) = C_{C_4}(u)$ with $u \in C_4$ a long root element (see Proposition \ref{p:graph}). As explained in the proof of \cite[Lemma 3.18]{BGS}, we find that $y_1$ is contained in the $G$-class $(A_1^3)^{(1)}$, while $y_2$ is in the class labelled $A_1^4$. Hence for $y \in (A_1^3)^{(1)}$ we deduce that 
$\dim(y^G \cap H^0) = 0$ and $\dim (y^G \cap (H \setminus H^0)) \leqs 27$, which implies that $\a(G,H,y) \leqs 43/70$. Similarly, for $y \in A_1^4$ we get $\a(G,H,y) \leqs 1/2$ if $p=2$, otherwise $\a(G,H,y) = 0$. 

Now assume $H = A_2A_5.2$, $p=2$ and $y \in H \setminus H^0$ is an involution. Here $y$ induces a graph automorphism on both simple factors of $H^0$ and we deduce that there are two $H^0$-classes to consider, represented by $y_1$ and $y_2$, where $C_{H^0}(y_1) = B_1C_3$ and $C_{H^0}(y_2) = B_1C_{C_3}(u)$ with $u \in C_3$ a long root element. As explained in the proof of \cite[Lemma 3.12]{BGG1}, the element $y_2$ is contained in the $G$-class $A_1^4$ and we claim that $y_1$ is in the class $(A_1^3)^{(2)}$. To see this, let $V$ be the $56$-dimensional Weyl module $W_G(\l_7)$  and observe that
\[
V\downarrow H^0 = (U_1 \otimes U_2) \oplus (U_1^* \otimes U_2^*) \oplus (0 \otimes \Lambda^3(U_2)),
\]
where $U_1$ and $U_2$ denote the natural modules for $A_2$ and $A_5$, respectively, and $0$ is the trivial module for $A_2$ (see \cite[Table 4]{Thomas}, for example). Now $y_1$ interchanges the first two summands and it has Jordan form $(J_2^6,J_1^8)$ on the final summand. Therefore, $y_1$ has Jordan form $(J_2^{24},J_1^8)$ on $V$ and by inspecting \cite[Table 7]{Law1} we conclude that $y_1$ is in the class $(A_1^3)^{(2)}$ as claimed. As a consequence, we deduce that $\dim(y^G \cap H) = \dim(y^G \cap H^0) = 20$ if $y \in (A_1^3)^{(2)}$ (for all $p$) and thus $\a(G,H,y) = 23/45$. Similarly, if $y \in A_1^4$ then either $p \geqs 3$ and $\a(G,H,y) = 7/15$, or $p=2$ and $\a(G,H,y) \leqs 1/2$.

Next suppose $H = T_1E_6.2$. Clearly, every unipotent element in the connected component $H^0$ is contained in the $E_6$ factor and the corresponding classes in $E_6$ and $G$ have the same label (note that the $E_6$-class labelled $A_1^3$
meets a Levi subgroup $A_6T_1$ of $G$, so it is contained in the $G$-class $(A_1^3)^{(2)}$). This allows us to compute $\dim(y^G \cap H^0)$ and so to complete the analysis of this case we may assume $p=2$ and $y \in H \setminus H^0$ is an involution. Here $y$ induces a graph automorphism on the $E_6$ factor and it inverts the $1$-dimensional central torus $T_1$. By Proposition \ref{p:graph}, there are two $H^0$-classes of involutions of this form, represented by $y_1$ and $y_2$, where $C_{H^0}(y_1) = F_4$ and $C_{H^0}(y_2) = C_{F_4}(u)$, with $u \in F_4$ a long root element. As explained in the proof of \cite[Lemma 4.1]{LLS1}, we calculate that each $y_i$ has Jordan form $(J_2^{28})$ on $V = W_G(\l_7)$ and by inspecting \cite[Table 7]{Law1} we deduce that each $y_i$ is in one of the $G$-classes labelled $(A_1^3)^{(1)}$ or $A_1^4$. In fact, the proof of \cite[Lemma 4.1]{LLS1} shows that $y_2$ is in $A_1^4$. So for $y \in (A_1^3)^{(1)}$ we get $\dim(y^G \cap H^0) = 0$ and $\dim(y^G \cap (H \setminus H^0)) \leqs 27$, which implies that $\a(G,H,y) \leqs 1/2$.  And similarly, if $y \in A_1^4$ then either $p \geqs 3$ and $\a(G,H,y) = 0$, or $p=2$, $\dim(y^G \cap H) \leqs 43$ and $\a(G,H,y) \leqs 1/2$.

Now let us turn to the case $H = A_1^3D_4.S_3$. Write $H^0 = H_1H_2$, where $H_1 = A_1$ and $H_2 = D_2D_4<D_6$ (here we are viewing $D_2$ as $A_1^2$). In particular, $H^0$ is contained in a maximal closed subgroup $A_1D_6$ and so we can appeal to \cite[Section 4.10]{Law2} in order to compute $\dim(y^G \cap H^0)$. 

For example, suppose $p=2$ and let us adopt the notation from \cite{AS} for unipotent involutions in orthogonal groups. Let $z_1,z_2 \in D_2$ be involutions of type $a_2$ and $c_2$, respectively, so $\dim z_1^{D_2} = 2$ and $\dim z_2^{D_2} = 4$. Then  it is straightforward to write down a set of representatives of the classes of involutions in $D_2D_4$ and we can easily identify the corresponding class in $D_6$, recalling that an involution of the form $uv \in D_2D_4$ is of type $a$ in $D_6$ if and only if $u$ and $v$ are both $a$-type involutions (or if one of them is type $a$ and the other is trivial). We can then use \cite[Section 4.10]{Law2} to identify the $G$-class of each involution $y \in H^0$ and this allows us to compute the dimension of $y^G \cap H^0$ as follows:
\[
\begin{array}{l|ccccc}
\mbox{$G$-class of $y$} & A_1 & A_1^2 & (A_1^3)^{(1)} & (A_1^3)^{(2)} & A_1^4 \\ \hline 
\dim(y^G \cap H^0) & 10 & 12 & 14 & 16 & 22 
\end{array}
\]

To complete the analysis of this case, let us assume $p \in \{2,3\}$ and $y \in H \setminus H^0$ has order $p$. Let $V$ be the Weyl module $W_G(\l_7)$. 

First assume $p=3$. Here $y$ induces a triality graph automorphism on the $D_4$ factor of $H^0$ and it cyclically permutes the three $A_1$ factors. Therefore, $C_{H^0}(y) = A_1G_2$ or $A_1C_{G_2}(u)$, where $u \in G_2$ is a long root element, and thus $\dim(y^G \cap (H \setminus H^0)) \leqs 26$. Now by considering the restriction of $V$ to $H^0$, it is straightforward to show that $y$ has Jordan form $(J_3^{18},J_1^2)$ on $V$ and therefore $y$ is in one of the classes labelled $A_2^2$ or $A_2^2A_1$. In particular, if  $y$ is in the class $A_2^2A_1$, then $\dim(y^G \cap H^0) = 0$ and thus $\dim(y^G \cap H) \leqs 26$, which yields $\a(G,H,y) \leqs 1/3$. 

Now assume $p=2$. Here $y$ induces a graph automorphism on $D_4$ (of type $b_1$ or $b_3$ in the notation of \cite{AS}) and it interchanges two of the $A_1$ factors. Therefore, we have at most four $H^0$-classes of involutions to consider, represented by the elements
\[
y_1 = (b_1,1), \, y_2 = (b_3,1), \, y_3 = (b_1,z), \, y_4 = (b_3,z),
\]
where the first component indicates the type of graph automorphism of $D_4$ induced by $y_i$ and the second indicates the action on the fixed $A_1$ factor, which is either trivial or an involution $z$. Now 
\begin{align*}
V \downarrow H^0 =  & \, (U_1 \otimes U_2 \otimes U_3 \otimes 0) \oplus (U_1 \otimes 0 \otimes 0 \otimes W_{D_4}(\omega_1)) \oplus (0 \otimes U_2 \otimes 0 \otimes W_{D_4}(\omega_3)) \\
& \, \oplus (0 \otimes 0 \otimes U_3 \otimes W_{D_4}(\omega_4)), 
\end{align*}
where $U_i$ is the natural module for the $i$-th $A_1$ factor, $0$ is the trivial module and the $\omega_i$ are fundamental dominant weights for $D_4$ (see \cite[Table 4]{Thomas}). Using this decomposition, we calculate that $y_1$ and $y_2$ have respective Jordan forms $(J_2^{20},J_1^{16})$ and $(J_2^{24},J_1^8)$ on $V$, while $y_3$ and $y_4$ have Jordan form $(J_2^{28})$. By inspecting \cite[Table 7]{Law1}, we deduce that  $y_1$ and $y_2$ are respectively contained in the $G$-classes $A_1^2$ and $(A_1^3)^{(2)}$, while $y_3^G, y_4^G \in \{ (A_1^2)^{(1)}, A_1^4\}$. Therefore, if $y \in A_1^2$ then $\dim(y^G \cap H) = 12$ (for all $p$) and thus $\a(G,H,y) =  7/12$. Similarly, if $y \in (A_1^3)^{(2)}$ then $\dim(y^G \cap H) \leqs 18$  and $\a(G,H,y) \leqs 25/48$, while for $y \in (A_1^3)^{(1)}$ we deduce that $\dim (y^G \cap H) \leqs 20$, which yields $\a(G,H,y) \leqs 31/48$. Similarly, if $y \in A_1^4$ then $\dim(y^G \cap H) = \dim(y^G \cap H^0) = 18+4\delta_{2,p}$ and thus $\a(G,H,y) = (11+\delta_{2,p})/24$. 

Next assume $H = A_1^7.{\rm GL}_3(2)$. Here $\dim(y^G \cap H^0) \leqs 14$ for all $y \in H^0$ and it is straightforward to check that $\dim(y^G \cap (H\setminus H^0)) \leqs 14$ for the relevant unipotent classes we are interested in. Note that if $p=7$ and $y \in H \setminus H^0$, then $y$ acts as a $7$-cycle on the $A_1$ factors of $H^0$ and by considering the restriction of $V = W_G(\l_7)$ to $H^0$ we deduce that $y$ has Jordan form $(J_7^8)$ on $V$, which places $y$ in the $G$-class labelled $A_6$ (see \cite[Table 7]{Law1}) and this is not one of the classes we are interested in. Therefore, $\dim(y^G \cap H) \leqs 14$ and we immediately obtain  the corresponding upper bound on $\a(G,H,y)$ in Table \ref{tab:beta_e7} unless $p \geqs 3$ and $y$ is contained in the class labelled $A_1^4$. 

So to complete the argument for $H = A_1^7.{\rm GL}_3(2)$, let us assume $y \in A_1^4$ and $p \geqs 3$. First we claim that $y^G \cap H = y^G \cap H^0$. To see this, first observe that if $p=3$ and $x \in H \setminus H^0$ has order $3$, then $x$ induces a permutation of the $A_1$ factors of $H^0$ with cycle-shape $(3^2,1)$ (for example, this follows by considering the summands arising in the decomposition of $W \downarrow H^0$, where $W = W_G(\l_1)$ is the adjoint module for $G$; see \cite[Table 4]{Thomas}). This implies that  $x$ has at least $16$ Jordan blocks of size $3$ on $V = W_G(\l_7)$ and thus $x$ is not in the class $A_1^4$. Similarly, we noted above that if $p=7$ and $x \in H \setminus H^0$ has order $7$, then $x$ is contained in the $G$-class $A_6$. So in order to establish the bound $\a(G,H,y) \leqs 27/56$, we need to show that $\dim(y^G \cap H^0) \leqs 12$. To do this, write $H^0 = H_1H_2$, where $H_1 = A_1$ and $H_2 = D_2^3<D_6$, so we have $H^0 < A_1D_6 < G$ and we can determine the $G$-class of each unipotent $H^0$-class by appealing to \cite[Section 4.12]{Law2}. In this way, it is straightforward to check that if $x = x_1 \cdots x_7 \in H^0$ has order $p$ and each $x_i$ is nontrivial, then $x$ is contained in the $G$-class labelled $A_2A_1^3$. In particular, if $y \in A_1^4$ then $\dim(y^G \cap H^0) \leqs 12$ as required.

Finally, let us assume $H = T.W$. Here the trivial bound $\dim(y^G \cap H) \leqs 7$ is sufficient unless $y \in A_1^4$ and $p \geqs 3$. In the latter case, we claim that $\a(G,H,y) = 0$. To see this, we may assume $p \in \{3,5,7\}$ and $x \in H \setminus H^0$ has order $p$ (recall that $W = 2 \times {\rm Sp}_6(2)$ and so the prime divisors of $|W|$ are $2,3,5$ and $7$). Here $x$ induces a permutation of order $p$ on the set of $1$-dimensional root spaces in the Lie algebra $V = W_G(\l_1)$ and we deduce that $x$ admits a Jordan block of size $p$ in its action on $V$. So by inspecting \cite[Table 8]{Law1}, we may assume $p=3$. We claim that $x$ has at least $32$ Jordan blocks of size $3$ on $V$, which is incompatible with containment in the $A_1^4$ class. To see this, first note that the action of ${\rm Sp}_6(2) < W$ on the set of root pairs $\{\pm \a\}$ in the root system of $G$ is permutation isomorphic to the action on nonzero vectors in the natural module for ${\rm Sp}_6(2)$. It follows that any element of order $3$ in $W$ has $2^4-1$, $2^2-1$ or $2^0-1$ fixed points on this set, so it must move at least $48$ root pairs and hence at least $96$ roots, which justifies the claim. The result follows.
\end{proof}

In order to prove Theorem \ref{t:main} for $G = E_7$, we may assume $t \in \{2,3,4\}$ since $\Delta$ is always non-empty if $t \geqs 5$ by \cite[Theorem 7]{BGG1}.

\begin{thm}\label{t:e7_4}
The conclusion to Theorem \ref{t:main} holds when $G = E_7$ and $t=4$.
\end{thm}

\begin{proof}
By Corollary \ref{c:fixV} we know that $\Delta$ is empty when $\mathcal{C}_i = A_1$ for all $i$, so it remains to show that $\Delta$ is non-empty in all other cases. As before, it suffices to verify the bound $\Sigma_X(H)<3$ for all $H \in \mathcal{M}$ and by appealing to  \cite[Theorem 3.1]{BGG1} we deduce that $\a(G,H,y) \leqs 7/9$ if $y \in A_1$, otherwise $\a(G,H,y)<2/3$. Therefore, $\Sigma_X(H) < 3\cdot 7/9+2/3 = 3$ and the result follows.
\end{proof}

Now assume $t \in \{2,3\}$. First we handle the special cases in Table \ref{tab:special}.

\begin{prop}\label{p:e7_00}
The set $\Delta$ is empty if $X$ is one of the cases in Table \ref{tab:special}.
\end{prop}

\begin{proof}
Write $\mathcal{C}_i = y_i^G$ and first assume $t=3$. Since the $G$-class $A_1^2$  is contained in the closure of $(A_1^3)^{(1)}$ (see \cite[Chapter 4]{Spal}), it suffices to show that $\Delta$ is empty when $(\mathcal{C}_1,\mathcal{C}_2,\mathcal{C}_3)$ is the triple $(A_1,(A_1^3)^{(1)},(A_1^3)^{(1)})$. 

To handle this case, first observe that we may embed each $y_i$ in a subgroup $L = D_6$ with $M = N_G(L) = A_1D_6$. More precisely, if $W$ denotes the natural module for $L$, then we may assume $y_1$ has Jordan form $(J_2^2,J_1^8)$ on $W$, while $y_2$ and $y_3$ both have Jordan form $(J_2^6)$ (see \cite[Section 4.10]{Law2} and  note that if $p=2$ then $y_1$ is an involution of type $a_2$, while $y_2$ and $y_3$ are both of type $a_6$, in the notation of \cite{AS}). In addition, since $\sum_i \dim C_W(y_i) = 22 < 2\dim W$, \cite[Theorem 7]{BGG2} implies that we may assume the $y_i$ topologically generate $L$. Set $\mathcal{D}_i = y_i^M$ and note that $\dim \mathcal{C}_1 = 34$, $\dim \mathcal{C}_i = 54$, $\dim \mathcal{D}_1 = 18$ and $\dim \mathcal{D}_i = 30$ for $i=2,3$. In particular, we compute $\dim X =  142$ and $\dim Y = 78$, where $Y = \mathcal{D}_1 \times \mathcal{D}_2 \times \mathcal{D}_3$, whence
\[
\dim X - \dim Y = 64 = \dim G - \dim M
\]
and thus Proposition \ref{p:fibre} implies that $\Delta$ is empty.

For the remainder, let us assume $t=2$. First assume $(\mathcal{C}_1,\mathcal{C}_2) = (A_1,(A_5)^{(1)})$, so $\dim \mathcal{C}_1 = 34$ and $\dim \mathcal{C}_2 = 102$. Here we may embed $y_1$ and $y_2$ in $L = D_6$, where $M = N_G(L) = A_1D_6$ and the $y_i$ have respective Jordan forms $(J_2^2,J_1^8)$ and $(J_6^2)$ on the natural module for $L$. In view of \cite[Theorem 7]{BGG2}, we may assume $y_1$ and $y_2$ topologically generate $L$ and we note that $\dim y_i^M = 18,54$ for $i=1,2$, respectively. The result now follows by applying Proposition \ref{p:fibre}. 

For the remaining cases, by considering closures it suffices to show that $\Delta$ is empty when $\mathcal{C}_1 = (A_1^3)^{(1)}$ and $\mathcal{C}_2$ is either 
$(A_3A_1)^{(1)}$ or $A_2A_1^3$. If $\mathcal{C}_2 = (A_3A_1)^{(1)}$ then we can proceed as above, embedding the $y_i$ in $L = D_6$ with respective Jordan forms $(J_2^6)$ and $(J_4^2,J_2^2)$ on the natural module. We leave the reader to check the details. 

Now suppose $\mathcal{C}_2 = A_2A_1^3$ and note that $\dim \mathcal{C}_1 = 54$ and $\dim \mathcal{C}_2 = 84$. To handle this case we need to modify the standard approach via Proposition \ref{p:fibre} because we cannot embed $y_2$ in a $D_6$ subgroup. First observe that we may embed $y_1$ and $y_2$ in $M = N_G(L) = A_1D_6$, where $y_1 \in L = D_6$ has Jordan form $(J_2^6)$ on the natural module for $L$ and $y_2 = u_2v_2 \in A_1D_6$ with $u_2$ in the $A_1$ factor of order $p$ and $v_2 \in L$ with Jordan form $(J_3^3,J_1^3)$. By applying \cite[Theorem 7]{BGG2}, we see that $L$ is topologically generated by $y_1$ and $v_2$, so if we set $y = (y_1,y_2) \in X$ then
\[
G(y) \leqs \overline{\la y_1,u_2,v_2 \ra} = \la L, u_2 \ra < M.
\]
Since $\la L, u_2 \ra^0 = L$ and $G(y)$ projects onto $L$, it follows that $G(y)^0 = L$. Similarly, if we set $\mathcal{D}_i = y_i^M$ and $Y = \mathcal{D}_1 \times \mathcal{D}_2$ then $G(z)^0 \leqs L$ for all $z \in Y$. Finally, we compute $\dim \mathcal{D}_1 = 30$ and $\dim \mathcal{D}_2 = 44$, so $\dim X- \dim Y = 64 = \dim G - \dim M$ and we now conclude by applying Proposition \ref{p:fibre}.
\end{proof}

\begin{rem}\label{r:e7_lie}
It is worth noting that several cases in Table \ref{tab:special} can be handled by arguing as in Remark \ref{r:f4_lie}, using the fact that $\Delta$ is non-empty only if the bound in \eqref{e:lieX} is satisfied (see \cite[Corollary 3.20]{BGG2} and note that the centre of the Lie algebra of $G = E_7$ is trivial if $p \geqs 3$). For example, if $t=2$, $(\mathcal{C}_1,\mathcal{C}_2) = ((A_1^3)^{(1)}, A_2A_1^3)$ and $p \geqs 3$, then $\dim X = 138$ and thus $\Delta$ is empty. However, if $t=2$, $(\mathcal{C}_1,\mathcal{C}_2) = ((A_1^3)^{(1)}, (A_3A_1)^{(1)})$ and $p \geqs 5$, then $\dim X = 140$ and so this approach via \cite[Corollary 3.20]{BGG2} is inconclusive.
\end{rem}

\begin{thm}\label{t:e7_3}
The conclusion to Theorem \ref{t:main} holds when $G = E_7$ and $t=3$. 
\end{thm}

\begin{proof}
Set $\mathcal{C}_i = y_i^G$. By applying Corollary \ref{c:fixV} and Proposition \ref{p:e7_00}, we see that $\Delta$ is empty for the cases in Tables \ref{tab:main} and \ref{tab:special}. So to complete the proof, it suffices to show that $\Sigma_X(H)<2$ for all $H \in \mathcal{M}$ and for each of the remaining possibilities for $X$. By Theorem \ref{t:par}, this holds if $H \in \mathcal{P}$ so we only need to consider the subgroups in $\mathcal{R}$. 

If $\mathcal{C}_i \ne A_1$ for all $i$, then \cite[Theorem 3.1]{BGG1} implies that $\Sigma_X(H) <2$ since we have $\a(G,H,y_i)<2/3$ for each $i$. Therefore, for the remainder we may assume $\mathcal{C}_1 = A_1$ and $H \in \mathcal{L}$ (indeed, if $H \not\in \mathcal{L}$ then $\a(G,H,y_1) = 0$ and thus $\Sigma_X(H) < 4/3$). 

First assume $\mathcal{C}_2 = A_1$. By inspecting \cite[Chapter 4]{Spal}, we see that the class labelled $A_2A_1^2$ is contained in the closure of $\mathcal{C}_3$. Therefore, in view of Proposition \ref{p:clos}, it suffices to show that $\Sigma_X(H)<2$ when $\mathcal{C}_3 = A_2A_1^2$ and it is easy to check that the bounds in Proposition \ref{p:e7_dim3} are sufficient. Similarly, if $\mathcal{C}_2 = A_1^2$ then we may assume $\mathcal{C}_3 = (A_1^3)^{(2)}$ and once again the bounds in Proposition \ref{p:e7_dim3} are good enough. Finally, if $\mathcal{C}_2 \ne A_1,A_1^2$ then by considering closures we may assume 
$(\mathcal{C}_2, \mathcal{C}_3) = ((A_1^3)^{(1)}, (A_1^3)^{(2)})$ or $((A_1^3)^{(2)}, (A_1^3)^{(2)})$. As before, the result follows via Proposition \ref{p:e7_dim3}.
\end{proof}

\begin{thm}\label{t:e7_2}
The conclusion to Theorem \ref{t:main} holds when $G = E_7$, $t=2$ and $p \geqs 3$.
\end{thm}

\begin{proof}
Write $\mathcal{C}_i = y_i^G$. If $(\mathcal{C}_1,\mathcal{C}_2)$ is one of the cases in Tables \ref{tab:main} or \ref{tab:special}, then $\Delta$ is empty by Corollary \ref{c:fixV} and Proposition \ref{p:e7_00}. So by appealing to Proposition \ref{p:fix} and Theorem \ref{t:par}, we just need to verify the bound $\Sigma_X(H) <1$ in each of the remaining cases, where $H$ is an arbitrary subgroup in the collection $\mathcal{R}$. 

First assume $\mathcal{C}_1 = A_1$. If $H \not\in \mathcal{L}$ then $\Sigma_X(H) < 2/3$ by \cite[Theorem 3.1]{BGG1}, so we may assume $H \in \mathcal{L}$. By considering the closure relation on unipotent classes, we may assume $\mathcal{C}_2$ is the class labelled $A_4A_2$ and the result now follows by applying the relevant upper bounds in Proposition \ref{p:e7_dim3}. 

Next assume $\mathcal{C}_1 = A_1^2$ or $(A_1^3)^{(1)}$. Here it is sufficient to show that $\Delta$ is non-empty when $\mathcal{C}_2 = A_2^2A_1$. If $H \in \mathcal{L}$ then the bounds in Proposition \ref{p:e7_dim3} yield $\Sigma_X(H)<1$, so we may assume $H \not\in \mathcal{L}$. If $\mathcal{C}_i \cap H$ is empty for $i=1$ or $2$ then \cite[Theorem 3.1]{BGG1} gives $\Sigma_X(H)<2/3$ and so we can assume that both $\mathcal{C}_1$ and $\mathcal{C}_2$ have representatives in $H$. By inspecting \cite{Law2}, we can rule out $H = A_2.2$, $A_1^2$ and $A_1$, which leaves $H = A_1G_2$ and $(2^2 \times D_4).S_3$ and so there are two possibilities to consider.

If $H = A_1G_2$ then we use \cite[Section 5.9]{Law2} to show that $\Sigma_X(H) \leqs 27/29$. Now assume $H = (2^2 \times D_4).S_3$. Here $H^0 = D_4 < A_7<G$, where the $D_4<A_7$ is the standard embedding corresponding to the natural module for $D_4$. By appealing to \cite[Section 4.11]{Law2}, it is easy to compute $\dim(y^G \cap H^0)$ for each unipotent element $y \in G$. For the relevant classes we are interested in we get 
\[
\dim(y^G \cap H^0) = \left\{\begin{array}{ll}
10 & \mbox{if $y \in A_1^2$} \\
0 & \mbox{if $y \in (A_1^3)^{(1)}$ or $A_2^2A_1$.}
\end{array}\right.
\]
If $p=3$ and $y \in H \setminus H^0$ has order $3$, then $y$ acts as a triality graph automorphism on $H^0$ and by arguing as in the proof of \cite[Lemma 3.18]{BGG1} we deduce that $y$ has at least $35$ Jordan blocks of size $3$ on the adjoint module $W_G(\l_1)$. By inspecting \cite[Table 8]{Law1}, we conclude that $\a(G,H,y) = 3/5$ if $y \in A_1^2$ and $\a(G,H,y) = 0$ for $y \in (A_1^3)^{(1)}$. For $y \in A_2^2A_1$ we observe that $\dim(y^G \cap H) \leqs 20$, whence $\a(G,H,y) \leqs 1/3$ and $\Sigma_X(H) \leqs 3/5+1/3 <1$. 

Now suppose $\mathcal{C}_1 = (A_1^3)^{(2)}$ or $A_2$. Here we need to show that 
$\Sigma_X(H)<1$ for $\mathcal{C}_2 = A_2A_1^2$ and one can check that the bounds in Proposition \ref{p:e7_dim3} are effective for $H \in \mathcal{L}$. Now assume $H \not\in \mathcal{L}$. By inspecting \cite{Law2}, the problem is quickly reduced to the case where $H = (2^2 \times D_4).S_3$. By arguing as in the previous paragraph, we compute $\a(G,H,y) = 0$ for $y \in (A_1^3)^{(2)}$ and $\a(G,H,y) = 17/35$ for $y \in A_2$. And if $y \in A_2A_1^2$ we observe that $\dim(y^G \cap H^0) = 16$ and $\dim(y^G \cap (H \setminus H^0)) \leqs 20$, which yields $\a(G,H,y) \leqs 43/105$. Bringing these bounds together, we conclude that $\Sigma_X(H) \leqs 17/35+43/105<1$.

To complete the proof we may assume $\mathcal{C}_i \not\in \{A_1,A_1^2,(A_1^3)^{(1)},(A_1^3)^{(2)},A_2\}$ for $i=1,2$. Here it suffices to show that $\Delta$ is non-empty when $\mathcal{C}_1 = \mathcal{C}_2 = A_1^4$. Since $(2^2 \times D_4).S_3$ does not meet the class $A_1^4$, and similarly for the subgroups $A_1G_2$, $A_2.2$, $A_1^2$ and $A_1$, we may assume $H \in \mathcal{L}$. Recalling that $p$ is odd, one can now check that the bound on $\a(G,H,y)$ in Proposition \ref{p:e7_dim3} for $y \in A_1^4$ is sufficient in every case. 
\end{proof}

This completes the proof of Theorem \ref{t:main} for $G = E_7$.

\subsection{$G = E_8$}\label{ss:e8}

In order to complete the proof of Theorem \ref{t:main} we may assume $G = E_8$. We refer the reader to \cite[Table 22.1.1]{LS_book} for a list of the unipotent classes in $G$ and their respective dimensions. As usual, we follow \cite{LS_book} in our choice of notation for the unipotent classes in $G$, and we define $\mathcal{M}$, $\mathcal{P}$, $\mathcal{R}$ and $\mathcal{L}$ as in Section \ref{ss:sub}.

\begin{prop}\label{p:e8_dim3}
Let $H \in \mathcal{L}$ and let $y \in G$ be an element of order $p$ in one of the following conjugacy classes
\[
A_1, \, A_1^2, \, A_1^3, \, A_1^4, \, A_2, \, A_2A_1^3, \, A_2^2A_1^2, \, A_4A_3.
\]
Then $\a(G,H,y) \leqs \b$, where $\b$ is recorded in Table \ref{tab:beta_e8}.
\end{prop}

{\small
\begin{table}
\[
\begin{array}{l|cccccccc}
& A_1 & A_1^2 & A_1^3 & A_1^4 & A_2 & A_2A_1^3 & A_2^2A_1^2 & A_4A_3 \\ \hline
A_1E_7 & 11/14 & 9/14 & 4/7 & (27+\delta_{2,p})/56 & 4/7 & 3/8 & 9/28 & 0 \\
D_8 & 3/4 & 5/8 & 9/16 & (15+\delta_{2,p})/32 & 35/64 & 3/8 & 5/16 & 3/16 \\
G_2F_4 & 10/13 & 8/13 & 7/13 & 45/91 & 7/13 & 34/91 & 30/91 & 0 \\
A_2E_6.2 & 7/9 & 17/27 & 5/9 & (26+\delta_{2,p})/54 & 5/9 & 31/81 & 1/3 & 0 \\
A_8.2 & 3/4 & 13/21 & 23/42 & (20+\delta_{2,p})/42 & 1/2 & 31/84 & 0 & 4/21 \\
A_4^2.4 & 3/4 & 31/50 & 27/50 & (24+\delta_{2,p})/50 & 1/2 & 37/100 & 8/25 & 1/5 \\
D_4^2.(S_3 \times 2) & 3/4 & 5/8 & 9/16 & (15+\delta_{2,p})/32 & 17/32 & 13/32 & 1/3 & 0 \\
A_2^4.{\rm GL}_2(3) & 3/4 & 11/18 & 31/54 & (26+\delta_{2,p})/54 & 31/54 & 7/18 & 1/3 & 0 \\
A_1G_2^2.2 & 165/217 & 137/217 & 113/217 & 103/217 & 113/217 & 81/217 & 71/217 & 0  \\ 
A_1^8.{\rm AGL}_3(2) & 3/4 & 37/56 & 4/7 & (13+\delta_{2,p})/28 & 9/16 & 43/112 & 9/28 & 5/28 \\
T.W & 61/80 & 13/20 & 17/30 & \delta_{2,p}/2 & 67/120 & 47/120 & 1/3 & 1/5 \\
\end{array}
\]
\caption{The upper bound $\a(G,H,y) \leqs \b$ in Proposition \ref{p:e8_dim3}}
\label{tab:beta_e8}
\end{table}}

\begin{proof}
If $y \in A_1$ is a long root element, then $\a(G,H,y)$ is given in \cite[Table 1]{BGG1}, noting that $\a(G,H,y)=3/4$ when $H = D_4^2.(S_3 \times 2)$ and $p=2$. Indeed, in this case there are no long root elements in $H \setminus H^0$ (see below) and this allows us to deduce that $\a(G,H,y)=3/4$ for all $p$. For the remainder, we will assume $y \not\in A_1$. Let $V = W_G(\l_8)$ be the adjoint module for $G$.

Suppose $H \in \{A_1E_7, D_8, G_2F_4\}$. Here $H$ is connected and the $G$-class of each unipotent class in $H$ is determined in \cite{Law2}. From this we can  compute $\dim(y^G \cap H)$ and then obtain $\dim C_{\O}(y)$ via Proposition \ref{p:dim}.

Next assume $H = A_2E_6.2$. Here the $G$-class of each unipotent class in the connected component $H^0$ is recorded in \cite[Section 4.15]{Law2} and this allows us to compute $\dim(y^G \cap H^0)$. So to complete the analysis of this case, we may assume $p=2$ and $y \in H \setminus H^0$ is an involution. Here $y$ acts as a graph automorphism on both simple factors of $H^0$ and by applying Proposition \ref{p:graph} we deduce that  there are two $H^0$-classes of such elements, represented by $y_1$ and $y_2$, where $C_{H^0}(y_1) = B_1F_4$ and $C_{H^0}(y_2) = B_1C_{F_4}(u)$, with $u \in F_4$ a long root element. Then $\dim (y_i^G \cap (H \setminus H^0)) = 31, 47$ for $i=1,2$ and by considering the Jordan form of $y_1$ and $y_2$ on $V$ (see the proof of \cite[Lemma 3.11]{BGG1}) we see that $y_1$ is in the $G$-class labelled $A_1^3$, whereas $y_2$ is in $A_1^4$. If $y \in A_1^3$ then $\dim(y^G \cap H^0) = 40$ and thus $\a(G,H,y) = 5/9$. Similarly, if $y \in A_1^4$ then $\dim(y^G \cap H^0) = 44$, so $\dim(y^G \cap H) \leqs 44+3\delta_{2,p}$ and we deduce that $\a(G,H,y) \leqs (26+\delta_{2,p})/54$. 

The case $H = A_8.2$ is very similar, working with \cite[Section 4.16]{Law2} to compute $\dim(y^G \cap H^0)$. If $p=2$ and $y \in H \setminus H^0$ is an involution, then $y$ induces a graph automorphism on $H^0$, so $C_{H^0}(y) = B_4$, $\dim (y^G \cap (H\setminus H^0)) = 44$ and by arguing as in the proof of \cite[Proposition 5.11]{BGS} we deduce that $y$ is contained in the $G$-class $A_1^4$. Since $\dim (y^G \cap H^0) = 40$, we conclude that $\dim(y^G \cap H) \leqs 40+4\delta_{2,p}$ and thus $\a(G,H,y) \leqs (20+\delta_{2,p})/42$ as claimed. 

Now suppose $H = A_4^2.4$. Once again, we can compute $\dim (y^G \cap H^0)$ by inspecting \cite{Law2}, so we may assume $p=2$ and $y \in H \setminus H^0$ is an involution. By considering the restriction of $V$ to $H^0$ (see \cite[Table 5]{Thomas}) we deduce that $y$ induces a graph automorphism on both $A_4$ factors of $H^0$, so $C_{H^0}(y) = B_2^2$, $\dim (y^G \cap (H \setminus H^0)) = 28$ and we calculate that $y$ is in the $G$-class labelled $A_1^4$ (see the proof of \cite[Lemma 4.4]{BTh0}, for example). Since $\dim(y^G \cap H^0) = 24$, we conclude that $\a(G,H,y) \leqs (24+\delta_{2,p})/50$ when $y \in A_1^4$. 

Next let us turn to the case $H = D_4^2.(S_3 \times 2)$. Here $H^0 < D_8 < G$ and the embedding of $H^0$ in $D_8$ is transparent. Therefore, we can identify the $G$-class of each unipotent element in $H^0$ by appealing to \cite[Section 4.13]{Law2}. In turn, this allows us to compute $\dim (y^G \cap H^0)$ and so the analysis of this case is reduced to the situation where $p \in \{2,3\}$ and $y \in H \setminus H^0$ has order $p$. 

First assume $p=3$ and $y \in H \setminus H^0$ has order $3$. Here $y$ induces a triality graph automorphism on both $D_4$ factors and thus Proposition \ref{p:graph} implies that $\dim(y^G \cap (H \setminus H^0)) \leqs 40$. Next observe that
\[
V \downarrow H^0 = \mathcal{L}(H^0) \oplus (U_1 \otimes U_1) \oplus (U_2 \otimes U_2) \oplus (U_3 \otimes U_3)
\]
(see \cite[Table 5]{Thomas}), where $\mathcal{L}(H^0)$ is the Lie algebra of $H^0$ and $U_1$, $U_2$ and $U_3$ denote the three $8$-dimensional irreducible modules for $D_4$ (that is, the natural module and the two spin modules). Now $y$ cyclically permutes the three $64$-dimensional summands in this decomposition, which means that the Jordan form of $y$ on $V$ has at least $64$ Jordan blocks of size $3$. By inspecting \cite[Table 9]{Law1}, it follows that $y$ is not contained in any of the $G$-classes labelled $A_1, A_1^2, A_1^3, A_1^4$ or $A_2$. And if $y \in A_2A_1^3$ then the bound $\dim(y^G \cap H) \leqs 40$ yields $\a(G,H,y) \leqs 13/32$. Similarly, we get $\a(G,H,y) \leqs 1/3$ if $y \in A_2^2A_1^2$. 

Now assume $p=2$ and $y \in H \setminus H^0$ is an involution. Then up to conjugacy, $y$ either interchanges the two $D_4$ factors, or it acts as a graph automorphism on both factors. If $y$ swaps the two factors, then $y$ embeds in $D_8$ as an involution of type $(4A_1)'$ or $(4A_1)''$ in the notation of \cite[Table 8]{Law2} (that is, $y \in D_8$ is an involution of type $a_8$ or $a_8'$ in the notation of \cite{AS}) and by inspecting \cite[Section 4.13]{Law2} we deduce that $y$ is in the $G$-class $A_1^3$ or $A_1^4$. Similarly, if $y$ acts as a $b_1$-type graph automorphism on both factors, then $y \in D_8$ is of type $D_2$ and is therefore contained in the $G$-class $A_1^2$ (this corrects an error in the proof of \cite[Proposition 5.11]{BGS}, where it is incorrectly stated that $y$ is in the class $A_1$). And if $y$ acts as a $b_3$ graph automorphism on both factors, then $y$ is contained in the $D_8$-class $2A_1+D_2$ and is therefore in the $G$-class $A_1^4$. Similarly, if $y$ acts as a $b_1$ automorphism on one factor and $b_3$ on  the other, then $y$ is in the $G$-class $A_1^3$. So for $p=2$ we conclude that $\dim(y^G \cap (H \setminus H^0)) \leqs 30$ if $y \in A_1^4$ (since the class of $b_3$ graph automorphisms of $D_4$ has dimension $15$) and thus $\dim (y^G \cap H) = 26+6\delta_{2,p}$, which in turn yields $\a(G,H,y) = (15+\delta_{2,p})/32$. Similarly, if $y \in A_1^2$ then $\dim(y^G \cap H) = 20$ and $\a(G,H,y) = 5/8$. On the other hand, if $y \in A_1^3$ then $\dim(y^G \cap H) \leqs 28$ and we deduce that $\a(G,H,y) \leqs 9/16$.

Now suppose $H = A_2^4.{\rm GL}_2(3)$. Here $H^0 < A_2E_6$ and so we can work with the information in \cite[Sections 4.9, 4.15]{Law2} to compute $\dim(y^G \cap H^0)$. Now assume $p \in \{2,3\}$ and $y \in H \setminus H^0$ has order $p$. Suppose $p=2$. If $y \in A_1^2$ then the proof of \cite[Lemma 3.11]{BGG1} shows that $\dim(y^G \cap H) =8$ and we deduce that $\a(G,H,y) = 11/18$. For $y \in A_1^3,A_1^4$ we observe that $\dim (y^G \cap (H\setminus H^0)) \leqs 20$ (maximal if $y$ acts as a graph automorphism on each $A_2$ factor of $H^0$) and the result follows since $\dim(y^G \cap H^0) = 12,16$, respectively. Finally, suppose $p=3$. By arguing as in the proof of \cite[Lemma 3.11]{BGG1} we see that there are at least $54$ Jordan blocks of size $3$ in the Jordan form of $y$ on $V$ and so by inspecting \cite[Table 9]{Law1} we deduce that $y$ is not contained in any of the $G$-classes labelled $A_1$, $A_1^2$, $A_1^3$ or $A_1^4$. Moreover, by considering the action of $y$ on the simple factors of $H^0$, we see that $\dim(y^G \cap (H\setminus H^0)) \leqs 22$ (with equality if $y$ cyclically permutes three of the factors and acts as a regular element on the fixed factor). By comparing this estimate with $\dim(y^G \cap H^0)$, it follows that $\dim(y^G \cap H) \leqs 22$ for $y \in A_2, A_2A_1^3$, while $\dim(y^G \cap H) \leqs 24$ for $y \in A_2^2A_1^2$. In each case, this gives the bound $\a(G,H,y) \leqs \b$ presented in Table \ref{tab:beta_e8}.

Next assume $H = A_1G_2^2.2$ and note that $p \geqs 3$, so each $y \in H$ of order $p$ is contained in $H^0$. Since $H^0 = H_1H_2 < F_4G_2 < G$ with $H_1 = A_1G_2 < F_4$ and $H_2 = G_2$, we can use the information in \cite[Sections 5.3, 5.12]{Law2} to compute $\dim(y^G \cap H)$ and the result follows. 

Now suppose $H = A_1^8.{\rm AGL}_3(2)$ and note that $\dim y^{H^0} \leqs 16$ if $y \in H^0$. If $p \in \{2,3,7\}$ and $y \in H \setminus H^0$ then $y$ induces a nontrivial permutation of the $A_1$ factors of $H^0$ and it is straightforward to check that $\dim y^{H^0} \leqs 16$ for the elements of order $p$ we need to consider. For example, if $y$ has cycle-shape $(3^2,1^2)$ on the set of $A_1$ factors, then $\dim y^{H^0} \leqs 6+6+4 = 16$, with equality if $y$ acts nontrivially on the two fixed factors. And if $p=7$ and $y$ induces a $7$-cycle on the $A_1$ factors, then by considering the restriction of $V = W_G(\l_8)$ to $H^0$ we deduce that $y$ has at least $32$ Jordan blocks of size $7$ on $V$, which is incompatible with the Jordan form of the elements we are interested in (see \cite[Table 9]{Law1}). We conclude that $\dim(y^G \cap H) \leqs 16$ and this gives the bound $\a(G,H,y)\leqs \b$ in Table \ref{tab:beta_e8} unless $y \in A_1^4$ and $p \geqs 3$. In the latter case, we can argue as in the proof of Proposition \ref{p:e7_dim3} (for the case $H = A_1^7.{\rm GL}_3(2)$) to show that $y^G \cap H = y^G \cap H^0$ and $\dim(y^G \cap H) = 8$, which yields $\a(G,H,y) = 13/28$.

Finally, if $H = T.W$ then the trivial bound $\dim (y^G \cap H) \leqs 8$ is sufficient unless $y \in A_1^4$ and $p \geqs 3$. In the latter case, by arguing as in the proof of Proposition \ref{p:e7_dim3}, we deduce that $y^G \cap H$ is empty and thus $\a(G,H,y) = 0$. This completes the proof of the proposition.
\end{proof}

We are now ready to prove Theorem \ref{t:main} for $G = E_8$. In view of \cite[Theorem 7]{BGG1}, we may assume $t \in \{2,3,4\}$. First we handle the cases recorded in Table \ref{tab:special}.

\begin{prop}\label{p:e8_00}
The set $\Delta$ is empty if $X$ is one of the cases in Table \ref{tab:special}.
\end{prop}

\begin{proof}
Write $\mathcal{C}_i = y_i^G$ and first assume $t=4$, so $(\mathcal{C}_1, \mathcal{C}_2, \mathcal{C}_3, \mathcal{C}_4) = (A_1,A_1,A_1,A_1^2)$. Here we may embed each $y_i$ in a subgroup $L = E_7$ such that $M=N_G(L)=A_1E_7$ is a maximal closed subgroup of $G$. Set $\mathcal{D}_i = y_i^M = y_i^L$ and note that the $L$-class $\mathcal{D}_i$ and the $G$-class $\mathcal{C}_i$ have the same labels (see \cite[Section 4.14]{Law2}), so $\dim \mathcal{C}_i = 58$ and $\dim \mathcal{D}_i = 34$ for $i=1,2,3$, $\dim \mathcal{C}_4 = 92$ and $\dim \mathcal{D}_4 = 52$. By Theorem \ref{t:e7_4}, we may assume that the $y_i$ topologically generate $L$. It is now straightforward to check that all of the conditions in Proposition \ref{p:fibre} are satisfied and we conclude that $\Delta$ is empty.

Next assume $t=3$. By considering Proposition \ref{p:clos} and the closure relation on unipotent classes (see \cite{Spal}), it suffices to show that $\Delta$ is empty for $(\mathcal{C}_1,\mathcal{C}_2,\mathcal{C}_3) = (A_1,A_1,A_3)$ and $(A_1,A_1^2,A_2)$. Suppose $(\mathcal{C}_1,\mathcal{C}_2,\mathcal{C}_3) = (A_1,A_1,A_3)$. As before, we may embed each $y_i$ in $L = E_7$, where $M = N_G(L) = A_1E_7$ and the classes $\mathcal{C}_i$ and $\mathcal{D}_i = y_i^M = y_i^L$ have the same labels. Moreover, by applying Theorem \ref{t:e7_3}, we may assume that the $y_i$ topologically generate $L$. If we set $Y = \mathcal{D}_1 \times \mathcal{D}_2 \times \mathcal{D}_3$, then $\dim X = 264$, $\dim Y = 152$ and thus $\Delta$ is empty by Proposition \ref{p:fibre}. The case $(\mathcal{C}_1,\mathcal{C}_2,\mathcal{C}_3) = (A_1,A_1^2,A_2)$ is entirely similar.

For the remainder we may assume $t=2$. Suppose $\mathcal{C}_1 = A_1$. By considering Proposition \ref{p:clos} and the closure relation on unipotent classes in $G$, we may assume that $\mathcal{C}_2 \in \{ D_5, D_4A_2\}$. Suppose $\mathcal{C}_2 = D_5$. Here we may embed $y_1$ and $y_2$ in $L = E_7$ so that $M = N_G(L) = A_1E_7$ and the classes $\mathcal{D}_i = y_i^M = y_i^L$ and $\mathcal{C}_i$ have the same labels. In addition, we may assume that the $y_i$ topologically generate $L$ (see Theorem \ref{t:e7_2}) and we now apply Proposition \ref{p:fibre}, noting that the condition in part (iii) holds since $\dim X = 258$ and $\dim Y = 146$. 

Now assume $(\mathcal{C}_1,\mathcal{C}_2) = (A_1,D_4A_2)$. This case requires a slight variation of the usual argument (this is analogous to the case we considered in the final paragraph of the proof of Proposition \ref{p:e7_00}). First we embed the $y_i$ in $M = A_1E_7 = N_G(L)$, where $L = E_7$. More precisely, we take $y_1$ to be in the $A_1$-class of $L$ and we choose $y_2 = u_2v_2 \in A_1E_7$, where $u_2$ is an element of order $p$ in the $A_1$ factor and $v_2 \in L$ is contained in the $E_7$-class labelled $D_5(a_1)A_1$ (see \cite[Table 23]{Law2}). Set $\mathcal{D}_i = y_i^M$ and $Y = \mathcal{D}_1 \times \mathcal{D}_2$, so $\dim X = 256$ and $\dim Y = 144$. By Theorem \ref{t:e7_2}, we may assume that $y_1$ and $v_2$ topologically generate $L$. Setting $y = (y_1,y_2) \in X$, this implies that $G(y)^0 = L$ and we have $G(z)^0 \leqs L$ for all $z \in Y$. We now apply Proposition \ref{p:fibre} to conclude, noting that $\dim X - \dim Y = \dim G - \dim M$.

By considering the closure relation on unipotent classes, it remains to show that $\Delta$ is empty when $(\mathcal{C}_1,\mathcal{C}_2) = (A_1^2,D_4)$ or $(A_2,A_3)$. In both cases, this is a straightforward application of Proposition \ref{p:fibre}, where we embed each $y_i$ in $L = E_7$. We omit the details.
\end{proof}

\begin{thm}\label{t:e8_4}
The conclusion to Theorem \ref{t:main} holds when $G = E_8$ and $t=4$.
\end{thm}

\begin{proof}
By combining Corollary \ref{c:fixV} and Propositions \ref{p:fix} and \ref{p:e8_00}, we see that it suffices to show that $\Sigma_X(H)<3$ for all $H \in \mathcal{M}$ whenever $X$ is not one of the cases in Tables \ref{tab:main} and \ref{tab:special}. By Theorem \ref{t:par}, this inequality holds if $H \in \mathcal{P}$, so we may assume $H \in \mathcal{R}$.

Fix an element $y \in G$ of order $p$. If $H \ne A_1E_7$ then \cite[Theorem 3.1]{BGG1} states that $\a(G,H,y) \leqs 7/9$ if $y \in A_1$, otherwise $\a(G,H,y)<2/3$. Therefore, $\Sigma_X(H) < 3\cdot 7/9 + 2/3 = 3$ as required. Finally, if $H = A_1E_7$ then by considering closures we may assume 
\[
(\mathcal{C}_1,\mathcal{C}_2,\mathcal{C}_3,\mathcal{C}_4) = (A_1,A_1,A_1,A_1^3) \mbox{ or } (A_1,A_1,A_1^2,A_1^2),
\]
and in both cases we deduce that $\Sigma_X(H) < 3$ by inspecting Table \ref{tab:beta_e8}.
\end{proof}

\begin{thm}\label{t:e8_3}
The conclusion to Theorem \ref{t:main} holds when $G = E_8$ and $t=3$.
\end{thm}

\begin{proof}
Set $\mathcal{C}_i = y_i^G$. By Corollary \ref{c:fixV} and Proposition \ref{p:e8_00}, we know that  $\Delta$ is empty for each case in Tables \ref{tab:main} and \ref{tab:special}. Therefore, it remains to show that $\Delta$ is non-empty in all the other cases. As usual, to do this we will work with Proposition \ref{p:fix} and Theorem \ref{t:par}, which imply that it suffices to show that $\Sigma_X(H)<2$ for all $H \in \mathcal{R}$. 

If $\mathcal{C}_i \ne A_1$ for all $i$, then \cite[Theorem 3.1]{BGG1} gives $\a(G,H,y_i)<2/3$ and thus $\Sigma_X(H) <2$.  Therefore, we may assume $\mathcal{C}_1 = A_1$ and $H \in \mathcal{L}$, which brings the bounds in Proposition \ref{p:e8_dim3} into play.

First assume $\mathcal{C}_2 = A_1$. By considering Proposition \ref{p:clos} and the closure relation on unipotent classes, we may assume $\mathcal{C}_3 = A_2A_1^3$ and the desired bound $\Sigma_X(H)<2$ now follows via Proposition \ref{p:e8_dim3}. Similarly, if $\mathcal{C}_2 = A_1^2$ then we may assume $\mathcal{C}_3 = A_1^4$ and once again the bounds in Proposition \ref{p:e8_dim3} are good enough. Finally, if $\mathcal{C}_2 \ne A_1,A_1^2$ then by considering closures we may assume $\mathcal{C}_2 = \mathcal{C}_3 = A_1^3$ and the result follows by applying Proposition \ref{p:e8_dim3}. 
\end{proof}

\begin{thm}\label{t:e8_2}
The conclusion to Theorem \ref{t:main} holds when $G = E_8$, $t=2$ and $p \geqs 3$.
\end{thm}

\begin{proof}
Write $\mathcal{C}_i = y_i^G$. In the usual way, by combining Corollary \ref{c:fixV} and Proposition  \ref{p:e8_00}, we observe that $\Delta$ is empty for each pair $(\mathcal{C}_1,\mathcal{C}_2)$ in Tables \ref{tab:main} and \ref{tab:special}. In order to show that $\Delta$ is non-empty in each of the remaining cases, it is enough to verify the bound $\Sigma_X(H) <1$ for all $H \in \mathcal{R}$.  

First assume $\mathcal{C}_1 = A_1$. If $H \not\in \mathcal{L}$ then $\a(G,H,y_1) = 0$ and the desired inequality clearly holds, so we may assume $H \in \mathcal{L}$. By considering the possibilities for $\mathcal{C}_2$ and the closure relation on unipotent classes (see \cite[Chapter 4]{Spal}), it suffices to show that $\Sigma_X(H)<1$ for $\mathcal{C}_3 = A_4A_3$ and this follows immediately from the upper bounds in Proposition \ref{p:e8_dim3}. 

Next assume $\mathcal{C}_1 = A_1^2$. By considering closures, we may assume $\mathcal{C}_2 = A_2^2A_1^2$. If $H \in \mathcal{L}$ then the bounds in Proposition \ref{p:e8_dim3} imply that $\Sigma_X(H)<1$ as required. On the other hand, if $H \not\in \mathcal{L}$ then either $H \cap \mathcal{C}_1$ is empty (and thus $\Sigma_X(H) = \a(G,H,y_2) < 2/3$ by \cite[Theorem 3.1]{BGG1}), or by inspecting \cite{CST,Law2} we deduce that $H = F_4$ and $p=3$ (as noted in the proof of Theorem \ref{t:long}, if $H = A_1 \times S_5$ and $p \geqs 7$, then the elements of order $p$ in $H$ are contained in the class $E_8(a_7)$). For $(H,p) = (F_4,3)$, the information in \cite[Table 2]{CST} yields $\a(G,H,y_1) = 30/49$ and $\a(G,H,y_2) = 16/49$, whence $\Sigma_X(H) = 46/49$. 

Now suppose $\mathcal{C}_1 = A_1^3$ or $A_2$. In the usual way, by considering closures, we may assume $\mathcal{C}_2 = A_2A_1^3$. If $H \cap \mathcal{C}_1$ is empty then $\Sigma_X(H)<2/3$ by \cite[Theorem 3.1]{BGG1}, so we may assume $H$ meets the class $\mathcal{C}_1$. By arguing as in the previous paragraph, we deduce that $H \in \mathcal{L}$ and one can check that the upper bounds on $\a(G,H,y_i)$ in Proposition \ref{p:e8_dim3} are sufficient in all cases.

Finally, let us assume $\mathcal{C}_i \not\in \{A_1,A_1^2,A_1^3,A_2\}$ for $i=1,2$. Here the closure of $\mathcal{C}_i$ contains the class $A_1^4$, so we may assume $\mathcal{C}_1 = \mathcal{C}_2 = A_1^4$. If $H \in \mathcal{L}$ then the bounds in Proposition \ref{p:e8_dim3} are sufficient (recall that $p \geqs 3$). On the other hand, if $H \not\in \mathcal{L}$ then by arguing as above we see that we may assume $H = F_4$ and $p=3$, in which case the fusion information in \cite[Table 2]{CST} yields $\a(G,H,y) = 45/98$ for $y \in A_1^4$ and thus $\Sigma_X(H) = 45/49$.
\end{proof}

This completes the proof of Theorem \ref{t:main} for $G = E_8$, which in turn completes the proof of Theorem \ref{t:main} in full generality.

\newpage 

\section{The tables}\label{s:tab}

Here we present Tables \ref{tab:main} and \ref{tab:special} from Theorem \ref{t:main}. We refer the reader to Remark \ref{r:main} for comments on the notation for unipotent classes adopted in the tables.

\vs

{\small
\begin{table}[h]
\renewcommand{\thetable}{A}
\[
\begin{array}{lll} \hline
G & t & (\mathcal{C}_1, \ldots, \mathcal{C}_t) \\ \hline
G_2 & 3 & (A_1, A_1, A_1) \\
& 2 & (A_1,A_1), (A_1,\tilde{A}_1), (A_1,(\tilde{A}_1)_3), (A_1,G_2(a_1))  \\
& & \\
F_4 & 4 & (A_1, A_1, A_1, A_1) \\
& 3 & (A_1, A_1, A_1), (A_1, A_1, \tilde{A}_1), (A_1, A_1, (\tilde{A}_1)_2), (A_1, A_1, A_1\tilde{A}_1), (A_1, A_1, A_2)\\
& 2 & (A_1,A_1), (A_1,\tilde{A}_1), (A_1, A_1\tilde{A}_1), (A_1,A_2), (A_1, \tilde{A}_2), (A_1, A_2\tilde{A}_1), (A_1, \tilde{A}_2A_1), (A_1, B_2), (A_1,C_3(a_1)) \\
& & (A_1, F_4(a_3)), (A_1, B_3), (\tilde{A}_1, \tilde{A}_1), (\tilde{A}_1, A_1\tilde{A}_1), (\tilde{A}_1, A_2), (A_1\tilde{A}_1, A_1\tilde{A}_1), (A_1\tilde{A}_1, A_2), (A_2,A_2) \\
& & \\
E_6 & 4 & (A_1, A_1, A_1, A_1) \\
& 3 & (A_1, A_1, A_1), (A_1, A_1, A_1^2), (A_1, A_1, A_1^3), (A_1, A_1, A_2), (A_1, A_1^2, A_1^2) \\
& 2 & (A_1,A_1), (A_1,A_1^2), (A_1,A_1^3), (A_1,A_2), (A_1,A_2A_1), (A_1,A_2^2), (A_1,A_2A_1^2), (A_1,A_3), (A_1,A_2^2A_1) \\
& & (A_1,A_3A_1), (A_1,D_4(a_1)), (A_1,A_4), (A_1,D_4), (A_1^2,A_1^2), (A_1^2,A_1^3), (A_1^2,A_2), (A_1^2,A_2A_1) \\
& & (A_1^2,A_2A_1^2), (A_1^2,A_3), (A_1^3,A_1^3), (A_1^3,A_2), (A_2,A_2) \\ 
& & \\
E_7 & 4 & (A_1, A_1, A_1, A_1) \\ 
& 3 & (A_1,A_1,A_1), (A_1,A_1,A_1^2), (A_1,A_1,(A_1^3)^{(1)}), (A_1,A_1,(A_1^3)^{(2)}), (A_1,A_1,A_2), (A_1,A_1,A_1^4) \\
& &  (A_1,A_1,A_2A_1), (A_1,A_1^2,A_1^2) \\
& 2 & (A_1,A_1), (A_1,A_1^2), (A_1,(A_1^3)^{(1)}), (A_1,(A_1^3)^{(2)}), (A_1,A_2), (A_1,A_1^4), (A_1,A_2A_1), (A_1,A_2A_1^2) \\
& &  (A_1,A_2A_1^3), (A_1,A_2^2), (A_1,A_3), (A_1,(A_3A_1)^{(1)}), (A_1,A_2^2A_1), (A_1,(A_3A_1)^{(2)}), (A_1,D_4(a_1)) \\
& &  (A_1,A_3A_1^2), (A_1,D_4), (A_1,D_4(a_1)A_1), (A_1,A_3A_2), (A_1,A_4), (A_1,A_3A_2A_1), (A_1,D_4A_1) \\
& &  (A_1,A_4A_1), (A_1,D_5(a_1)), 
(A_1^2,A_1^2), (A_1^2,(A_1^3)^{(1)}), (A_1^2,(A_1^3)^{(2)}), (A_1^2,A_2), (A_1^2,A_1^4), (A_1^2,A_2A_1) \\
& & (A_1^2,A_2A_1^2), (A_1^2,A_2A_1^3),  
(A_1^2,A_3), ((A_1^3)^{(1)},(A_1^3)^{(2)}), ((A_1^3)^{(1)},A_2), ((A_1^3)^{(2)}, (A_1^3)^{(2)}) \\
& & ((A_1^3)^{(2)},A_2), ((A_1^3)^{(2)},A_1^4), ((A_1^3)^{(2)},A_2A_1), (A_2,A_2), (A_2, A_1^4), (A_2, A_2A_1) \\
& & \\
E_8 & 4 & (A_1, A_1, A_1, A_1) \\
& 3 & (A_1,A_1, A_1), (A_1,A_1,A_1^2), (A_1,A_1,A_1^3), (A_1,A_1,A_2), (A_1,A_1,A_1^4), (A_1,A_1^2,A_1^2) \\
& 2 & (A_1,A_1), (A_1,A_1^2), (A_1,A_1^3), (A_1,A_2), (A_1,A_1^4), (A_1,A_2A_1), (A_1,A_2A_1^2), (A_1,A_3), (A_1,A_2A_1^3) \\
& &  (A_1,A_2^2), (A_1,A_2^2A_1), (A_1,A_3A_1), (A_1,D_4(a_1)), (A_1,D_4), (A_1,A_2^2A_1^2), (A_1,A_3A_1^2) \\
& & (A_1, D_4(a_1)A_1), (A_1,A_3A_2), (A_1,A_4), (A_1,A_3A_2A_1), (A_1,D_4A_1),(A_1,D_4(a_1)A_2)  \\
& &  (A_1,A_4A_1), (A_1,A_3^2), (A_1^2, A_1^2), (A_1^2,A_1^3), (A_1^2,A_2), (A_1^2,A_1^4), (A_1^2,A_2A_1), (A_1^2,A_2A_1^2) \\
& & (A_1^2,A_3), (A_1^2,A_2A_1^3), (A_1^3, A_1^3), (A_1^3,A_2), (A_1^3,A_1^4), (A_2, A_2), (A_2,A_1^4) \\ \hline
\end{array}
\]
\caption{The varieties $X = \mathcal{C}_1 \times \cdots \times \mathcal{C}_t$ in Theorem \ref{t:main} with 
$\Delta = \emptyset$, Part I}
\label{tab:main}
\end{table}
}

{\small
\begin{table}[h]
\renewcommand{\thetable}{B}
\[
\begin{array}{lll} \hline
G & t & (\mathcal{C}_1, \ldots, \mathcal{C}_t) \\ \hline
F_4 & 3 & (A_1, \tilde{A}_1, \tilde{A}_1), (A_1, \tilde{A}_1, (\tilde{A}_1)_2), (A_1, (\tilde{A}_1)_2, (\tilde{A}_1)_2) \\
& 2 & (\tilde{A}_1,\tilde{A}_2), (\tilde{A}_1, A_2\tilde{A}_1), (\tilde{A}_1,B_2) \\
E_6 & 2 & (A_1^2,A_2^2) \\ 
E_7 & 3 & (A_1,A_1^2,(A_1^3)^{(1)}), (A_1,(A_1^3)^{(1)},(A_1^3)^{(1)}) \\
& 2 & (A_1,(A_5)^{(1)}), (A_1^2,A_2^2), (A_1^2,(A_3A_1)^{(1)}), ((A_1^3)^{(1)}, (A_1^3)^{(1)}), ((A_1^3)^{(1)},A_1^4), ((A_1^3)^{(1)},A_2A_1) \\
& &  ((A_1^3)^{(1)},A_2A_1^2), ((A_1^3)^{(1)},A_2A_1^3), ((A_1^3)^{(1)},A_2^2), ((A_1^3)^{(1)},A_3), ((A_1^3)^{(1)},(A_3A_1)^{(1)}) \\
E_8 & 4 & (A_1, A_1, A_1, A_1^2) \\  
& 3 & (A_1,A_1,A_2A_1), (A_1,A_1,A_2A_1^2), (A_1,A_1,A_3), (A_1,A_1^2,A_1^3), (A_1,A_1^2,A_2)  \\
& 2 & (A_1,D_5(a_1)), (A_1,A_4A_1^2), (A_1,A_4A_2), (A_1,A_4A_2A_1), (A_1,D_5(a_1)A_1), (A_1,A_5), (A_1,D_4A_2) \\
& &  (A_1,E_6(a_3)), (A_1,D_5), (A_1^2,A_2^2), (A_1^2,A_2^2A_1),(A_1^2,A_3A_1), (A_1^2,D_4(a_1)), (A_1^2,D_4), (A_1^3,A_2A_1) \\
& & (A_1^3,A_2A_1^2), (A_1^3,A_3), (A_2,A_2A_1), (A_2,A_2A_1^2), (A_2,A_3)  \\ \hline
\end{array}
\]
\caption{The varieties $X = \mathcal{C}_1 \times \cdots \times \mathcal{C}_t$ in Theorem \ref{t:main} with 
$\Delta = \emptyset$, Part II}
\label{tab:special}
\end{table}
}

\clearpage

\section{Conflict of interest}

The author declares that he has no conflict of interest.

\end{document}